\documentclass{amsproc}

\usepackage{mathrsfs}
\usepackage{amscd}
\usepackage{amsthm}
\usepackage{amssymb}
\usepackage{amsmath}
\usepackage{mathtools}
\usepackage[normalem]{ulem}
\usepackage{enumitem}
\usepackage{epsfig}
\usepackage{tikz}
\usepackage{tikz-cd}
\usetikzlibrary{decorations.markings,shapes.geometric,matrix,arrows,positioning,arrows.meta}
\usepackage{todonotes}
\usepackage[final]{microtype}
\usepackage[hidelinks]{hyperref}
\def\MR#1{} 

\tikzset{anchorbase/.style={baseline={([yshift=-0.5ex]current bounding box.center)}},
  int/.style={thick},
  cross line/.style={preaction={draw=white,line width=6pt,-}},
  wall/.style={thin,double,blue},
  middlearrow/.style={postaction=decorate,decoration={markings,mark=at
    position .55 with {\arrow{stealth};}}},
  middlearrowrev/.style={postaction=decorate,decoration={markings,mark=at
    position .55 with {\arrowreversed{stealth};}}},
  ev/.style={shape=rectangle, draw}
}

\newsavebox{\pullback}
\sbox\pullback{%
\begin{tikzpicture}%
\draw (0,0) -- (1ex,0ex);%
\draw (1ex,0ex) -- (1ex,1ex);%
\end{tikzpicture}}

\newsavebox{\pullbackvert}
\sbox\pullbackvert{%
\begin{tikzpicture}%
\draw (0,0) -- (0.707ex,-0.707ex);%
\draw (0.707ex,-0.707ex) -- (1.41ex,0ex);%
\end{tikzpicture}}

\newtheorem{theorem}{Theorem}[section]
\newtheorem{lemma}[theorem]{Lemma}
\newtheorem{proposition}[theorem]{Proposition}

\theoremstyle{definition}
\newtheorem{definition}[theorem]{Definition}
\newtheorem{example}[theorem]{Example}

\theoremstyle{remark}
\newtheorem{remark}[theorem]{Remark}

\numberwithin{equation}{section}

\newcommand{\C}{\mathbb{C}}
\newcommand{\R}{\mathbb{R}}

\newcommand{\Z}{\mathbb{Z}}
\newcommand{\Q}{\mathbb{Q}}
\newcommand{\F}{\mathbb{F}}
\newcommand{\Hom}{\textup{\text{Hom}}}
\newcommand{\Ext}{\textup{\text{Ext}}}

\newcommand{\Aut}{\textup{\text{Aut}}}
\newcommand{\Sym}{\textup{\text{Sym}}}
\newcommand{\Spec}{\operatorname{Spec}}
\newcommand{\prim}{\textnormal{prim}}
\newcommand{\rep}{\mathsf{rep}}
\newcommand{\Fun}{\textup{\text{Fun}}}
\newcommand{\id}{\textup{\text{id}}}
\newcommand{\op}{\textup{\text{op}}}

\newcommand{\coker}{\operatorname{coker}}
\newcommand{\pt}{\text{\textnormal{pt}}}
\newcommand{\Wald}{\mathcal{S}}
\newcommand{\RWald}{\mathcal{R}}
\newcommand{\Hall}{H}
\newcommand{\HallM}{M}
\newcommand{\Jor}{\text{\textnormal{Jor}}}
\newcommand{\nil}{\mathsf{nil}}
\newcommand{\lin}{\ell}
\newcommand{\plin}{\tilde{\ell}}
\newcommand{\spaces}{\mathscr{S}}
\newcommand{\Span}{\mathsf{Span}}
\newcommand{\sm}{\mathsf{Sm}}
\newcommand{\pr}{\mathsf{P}}
\newcommand{\Mat}{\mathsf{Mat}_*}
\newcommand{\ev}{\text{\textnormal{ev}}}
\newcommand{\can}{\text{\textnormal{can}}}

\newcommand{\coh}{\mathsf{coh}}
\newcommand{\tor}{\mathsf{tor}}
\newcommand{\git}{/\!\!/}

\newcommand{\nilMod}{{\text -}\mathsf{mod}^{\mathsf{nil}}}
\newcommand{\sst}{\textnormal{ss}}
\newcommand{\st}{\textnormal{st}}
\newcommand{\fr}{\textnormal{fr}}

\newcommand{\Vect}{\mathsf{V}\mathsf{ect}}
\newcommand{\vect}{\mathsf{vect}}

\newcommand{\comm}[1]{}

\newcommand{\cS}{\mathcal{S}}

\newcommand{\Sl}[1]{\spaces_{/#1}}
\newcommand{\FSl}[1]{\mathscr{F}_{/#1}}
\renewcommand{\AA}{\mathcal{A}}
\newcommand{\BB}{\mathcal{B}}
\newcommand{\CC}{\mathcal{C}}

\mathchardef\mhyphen="2D 

\begin{document}

\title[Hall algebras via $2$-Segal spaces]{Hall algebras via $2$-Segal spaces}

\author[B. Cooper]{Benjamin Cooper}
\address{Department of Mathematics, University of Iowa, 14 MacLean Hall, Iowa City, IA
52242-1419 USA}
\email{ben-cooper@uiowa.edu}

\author[M.\,B. Young]{Matthew B. Young}
\address{Department of Mathematics and Statistics, Utah State University,
Logan, Utah 84322 USA}
\email{matthew.young@usu.edu}
\thanks{The authors thank the organizers of Higher Segal Spaces and their Applications to Algebraic $K$-Theory, Hall Algebras, and Combinatorics (24w5266), consisting of Julie Bergner, Joachim Kock, Maru Sarazola and (from the planned 2020 workshop) Mark Penney. The authors also thank the workshop participants and BIRS staff. M.\,B.\,Y. is partially supported by National Science Foundation grant DMS-2302363 and a Simons Foundation Travel Support for Mathematicians grant (Award ID 853541).}

\date{\today}


\begin{abstract}
This is an introduction to Hall algebras from the perspective of $2$-Segal spaces or decomposition spaces, as introduced by Dyckerhoff and Kapranov and G\'{a}lvez-Carrillo, Kock and Tonks, respectively. We explain how linearizations of the $2$-Segal space arising as the Waldhausen $\mathcal{S}_{\bullet}$-construction of a proto-exact category recover various previously known Hall algebras. We use the $2$-Segal perspective to study functoriality of the Hall algebra construction and explain how relative variants of $2$-Segal spaces lead naturally to representations of Hall algebras.
\end{abstract}

\maketitle

\section*{Introduction}

A heroic mathematician once explained to the authors that a Hall algebra is
``something like a quantum group.''  Here we present a different
perspective: a Hall algebra is ``a linearization of a $2$-Segal space.''
Realizing the hero's interpretation can be hard work, as speaking precisely about Hall algebras
necessitates many minor variations and technicalities. Their interpretation is often
justified by the observation that one can tell the story in such a way so as
to conclude with an intrinsic construction of a quantum group. On the
other hand, a wimpy mathematician who defies a litany of technical hurdles
through abstraction, may miss the trees for the forest, but sometimes gains
new and often valuable insights by having chosen a different path.

$2$-Segal spaces were introduced by Dyckerhoff and Kapranov in
\cite{dyckerhoff2019} and decomposition spaces by G\'{a}lvez-Carrillo, Kock
and Tonks in \cite{galvez2018}. These notions are equivalent, something
unsurprising considering that, in each case, the authors made their
discoveries through a careful study of the associativity relation which is
satisfied by Hall algebras and incidence algebras, respectively.
The connection between Hall algebras and $2$-Segal spaces occurs when the $2$-Segal space arises from Waldhausen's $\Wald_{\bullet}$-construction. 

For a suitable category $\CC$, the
$\Wald_{\bullet}$-construction is---roughly speaking---a simplicial object $\Wald_{\bullet}(\CC)$, 
whose $n$-simplices correspond to composition series of length $n$ in
$\CC$ \cite{waldhausen1985}. The homotopy groups of the corresponding geometric realization agree
with the algebraic $K$-theory of $\CC$:
$$K_n(\CC) \simeq \pi_{n+1} |\Wald_{\bullet}(\CC)|.$$
For Waldhausen, an important goal was to show that a broad collection of
categories $\CC$, now known as Waldhausen categories, could be considered
suitable, thereby extending Quillen's definition of algebraic $K$-theory to a much 
wider setting. For our purposes, the crucial
observation of the aforementioned authors is that the simplicial object
$\Wald_{\bullet}(\CC)$ is a $2$-Segal space. Waldhausen's 
$\Wald_{\bullet}$-construction, once equipped with the $2$-Segal structure, 
contains in its shadow the Hall algebra and its friends.

The story of Hall algebras can be told in a way which
allows these $2$-Segal structures to appear organically. The Hall
algebra of a suitable category $\CC$ is a ring which is
spanned by the isomorphism classes of objects in $\CC$ together with a
product which is determined by the combinatorics of extensions between these
objects. The first Hall algebra, discovered independently by Steinitz \cite{steinitz1901} and Hall
  \cite{hall1959}, arises when $\CC$ is the category of finite $p$-abelian groups ($p$ prime). Steinitz and Hall interpreted this
  Hall algebra in terms of the ring of symmetric functions. Perhaps the
most common definition of Hall algebras found in the literature today,
introduced by Ringel \cite{ringel1990}, concerns abelian categories $\CC$
with finiteness properties similar to those of the category of $p$-abelian
groups. Ringel's main examples concerned quivers $Q$ and their abelian categories of finite dimensional representations $\rep_{\F_q}(Q)$ over a finite field $\F_q$. It was known from Gabriel's theorem \cite{gabriel1972} that, for 
connected quivers $Q$, the category $\rep_{\F_q}(Q)$ has finitely many isomorphism classes of indecomposable objects if and only if the underlying graph of $Q$ is a Dynkin diagram from the list $A_n$, $D_n$ or $E_6$, $E_7$, $E_8$.
In each case, Ringel identified the Hall algebra of $\CC=\rep_{\F_q}(Q)$ with the positive part of the quantum
group associated to $Q$ \cite{ringel1990b}:
$$H(\rep_{\F_q}(Q)) \simeq U_{\sqrt{q}}^+(\mathfrak{g}_Q).$$ 
This construction established the Hall algebra as a central object in the field of geometric representation theory. Later, Green introduced a coproduct on the Hall algebra and extended Ringel's isomorphism to one of bialgebras \cite{green1995}. Moreover, taking $Q$ to be the Jordan (one loop) quiver recovers the original Hall algebra of Steinitz and Hall. In his study of quantum groups, Lusztig built magnificently on
these ideas, viewing the Hall algebra as a convolution algebra of
locally constant functions on the moduli stack of the objects of $\CC$ \cite{lusztig2010}. These ideas were also central to Nakajima's construction of representations of quantum affine algebras \cite{nakajima2001}.

The convolution algebra incarnation of the Hall algebra sees only very
coarsely the structure of the moduli stacks involved in its
definition. Barwick proved that the Waldhausen
$\Wald_{\bullet}$-construction $\Wald_{\bullet}(\CC)$ occurs as the
Goodwillie derivative of the functor which associates to $\CC$
its moduli space of objects \cite{barwick}. From the perspective of this
paper, the central observation of Dyckerhoff--Kapranov and
G\'{a}lvez-Carrillo--Kock--Tonks is that the Waldhausen
$\Wald_{\bullet}$-construction, seen as a $2$-Segal space, suffices to
encode the Hall algebra. The geometric definition of Lusztig factors through the $\Wald_{\bullet}$-construction, as evidenced by
Section \ref{twosec}, which produces the Hall algebra of Ringel as a linearization of the $2$-Segal space
associated to the category.

That the ideas of Steinitz, 123 years ago, have a home within an additive
linearization of a Goodwillie derivative of the most complicated object
(moduli stack of objects) that we can associate to the simplest non-trivial
object (the Jordan quiver) is a heroic novelty of sorts.

\subsection*{What is included in this article} Section \ref{onesec} contains brief accounts of $2$-Segal
spaces and the Waldhausen $\Wald_{\bullet}$-construction. We work in the generality of proto-exact categories, a non-additive generalization of exact categories which allows for the realization of Hall algebras of so-called $\F_1$-linear categories. Section \ref{twosec} constructs from a $2$-Segal space $X_{\bullet}$ an associative algebra object in spans. Various linearizations are then discussed and it is explained, via examples, how this includes a number of Hall algebras which appear in the literature. This includes Ringel's Hall algebra of $\F_q$-linear abelian categories, their $\F_1$-analogues introduced by Szczesny and the cohomological and motivic Hall algebras of Kontsevich--Soibelman and Joyce.
The remainder of the article is concerned with using $2$-Segal spaces as an organizational tool for the study of Hall algebras. Section \ref{sec:functoriality} contains a discussion of the functoriality of the Hall algebra construction while Section \ref{sec:rel2Seg} explains the role of relative variants of $2$-Segal spaces in the representation theory of Hall
algebras, two topics for which this abstraction yields constructive
approaches. 

\subsection*{What is not included in this article}

Being a narrow discussion of a subject containing many technicalities, this article
inevitably gives few proofs and glosses over some off-putting details. Readers are referred to the
references and excellent surveys by Schiffmann \cite{schiffmann2012b}, who
focuses on representation theoretic aspects of finitary Hall algebras of
quivers and curves, and Dyckerhoff \cite{dyckerhoff2018}, who takes the $2$-Segal
perspective we follow in this article. The adornments of Hall algebras 
such as coproducts and the crucial role of the Euler pairing in the theory can be found elsewhere \cite{schiffmann2012b}. The role of Hall algebras in geometric representation theory is discussed only indirectly; see \cite{schiffmann2012c} for an introduction in this direction. Similarly, we do not discuss higher categorical generalizations of Hall algebras, such as Hall monoidal categories or algebra objects in the $(\infty,2)$-category of spans \cite{dyckerhoff2018,dyckerhoff2019}. We do not discuss other ideas which interface the $2$-Segal/decomposition space perspectives of the articles \cite{dyckerhoff2019, galvez2018}. 

\section{Recollections on $2$-Segal spaces}\label{onesec}
\label{sec:Recol2Seg}

We recall the definition of $2$-Segal spaces in a form convenient for applications to Hall algebras. The theory of $2$-Segal spaces is treated in detail in \cite{dyckerhoff2019}. See \cite{stern2024} for a brief introduction.

\subsection{$2$-Segal spaces}
\label{sec:2Seg}

Let $\Delta$ be the category of non-empty finite ordinals and their monotone maps. We often use without mention the canonical equivalence of $\Delta$ with its full subcategory on the standard ordinals $[n]=\{0,\dots, n\}$, $n \in \Z_{\geq 0}$.

Let $\spaces$ be a Quillen model category. We refer to homotopy pullbacks and pushouts simply as pullbacks and pushouts, respectively, and similarly for Cartesian and coCartesian diagrams. We write $\pt$ for the final object of $\spaces$. The primary example to keep in mind is the model category of small groupoids, in which weak equivalences are equivalences of categories and cofibrations are functors which are injective on objects. With minor changes, $\spaces$ can also be taken to be an $\infty$-category with finite limits.

As a warm up and for later reference, we recall the definition of a $1$-Segal space.

\begin{definition}[\cite{rezk2001}]
\label{def:1Segal}
A simplicial object $X_{\bullet}: \Delta^{\op} \rightarrow \spaces$ is \emph{$1$-Segal} if, for every $n \geq 2$ and $0 \leq i \leq n$, the square
\[
\begin{tikzpicture}[baseline= (a).base]
\node[scale=1.0] (a) at (0,0){
\begin{tikzcd}
X_n \arrow{d} \arrow{r} & X_{\{i,\dots, n\}} \arrow{d}\\
X_{\{0, \dots,i\}} \arrow{r} & X_{\{i\}}
\end{tikzcd}
};
\end{tikzpicture}
\]
is Cartesian.
\end{definition}

In the above diagram, we have written $X_n$ for $X_{[n]}$ and all morphisms are induced by the obvious inclusions of ordinals. Similar comments apply in what follows.

\begin{definition}[\cite{dyckerhoff2019}]
\label{def:2Segal}
A simplicial object $X_{\bullet}: \Delta^{\op} \rightarrow \spaces$ is \emph{$2$-Segal} if, for every $n \geq 3$ and $0 \leq i < j \leq n$, the square
\[
\begin{tikzpicture}[baseline= (a).base]
\node[scale=1.0] (a) at (0,0){
\begin{tikzcd}
X_n \arrow{d} \arrow{r} & X_{\{i,\dots, j\}} \arrow{d}\\
X_{\{0, \dots,i,j,\dots, n\}} \arrow{r} & X_{\{i,j\}}
\end{tikzcd}
};
\end{tikzpicture}
\]
is Cartesian.
\end{definition}

\begin{definition}[\cite{dyckerhoff2019}]
\label{def:unitalSegal}
A $2$-Segal simplicial object $X_{\bullet}: \Delta^{\op} \rightarrow \spaces$ is \emph{unital} if, for every $n \geq 2$ and $0 \leq i \leq n-1$, the square
\[
\begin{tikzpicture}[baseline= (a).base]
\node[scale=1.0] (a) at (0,0){
\begin{tikzcd}
X_{n-1} \arrow{d} \arrow{r} & X_n \arrow{d}\\
X_{\{i\}} \arrow{r} & X_{\{i,i+1\}}
\end{tikzcd}
};
\end{tikzpicture}
\]
is Cartesian.
\end{definition}

It turns out that unitality is implied by the $2$-Segal condition.

\begin{theorem}[{\cite{feller2021}}]
\label{thm:unital}
A $2$-Segal space is unital.
\end{theorem}

\subsection{Proto-exact categories}
\label{sec:protoExact}

In order to treat various flavours of Hall algebras which appear in the literature, we work with the following non-additive generalization of Quillen's exact categories \cite{quillen1973}.

\begin{definition}[{\cite[\S 2.4]{dyckerhoff2019}}]
A \emph{proto-exact category} $\CC$ is a pointed category, whose zero object is denoted by $0$, with two classes of morphisms, $\mathfrak{I}$ and $\mathfrak{D}$, called \emph{inflations} and \emph{deflations} and denoted by $\rightarrowtail$ and $\twoheadrightarrow$, respectively, such that the following properties hold:
\begin{enumerate}[label=(\roman*)]
\item For each object $U \in \CC$, the morphism $0 \rightarrow U$ is in $\mathfrak{I}$ and the morphism $U \rightarrow 0$ is in $\mathfrak{D}$.

\item Each class $\mathfrak{I}$ and $\mathfrak{D}$ is closed under composition and contains all isomorphisms.

\item A commutative square of the form
\begin{equation}
\label{eq:biCartDiag}
\begin{tikzpicture}[baseline= (a).base]
\node[scale=1] (a) at (0,0){
\begin{tikzcd}
U \arrow[two heads]{d} \arrow[tail]{r} & V \arrow[two heads]{d}\\
W \arrow[tail]{r} & X
\end{tikzcd}
};
\end{tikzpicture}
\end{equation}
is a Cartesian if and only if it is coCartesian.

\item Any diagram
\[
\begin{tikzpicture}[baseline= (a).base]
\node[scale=1] (a) at (0,0){
\begin{tikzcd}
& V \arrow[two heads]{d}\\
W \arrow[tail]{r} & X
\end{tikzcd}
};
\end{tikzpicture}
\]
can be completed to a biCartesian square of the form \eqref{eq:biCartDiag}.

\item Any diagram
\[
\begin{tikzpicture}[baseline= (a).base]
\node[scale=1] (a) at (0,0){
\begin{tikzcd}
U \arrow[two heads]{d} \arrow[tail]{r} & V\\
W &
\end{tikzcd}
};
\end{tikzpicture}
\]
can be completed to a biCartesian square of the form \eqref{eq:biCartDiag}.
\end{enumerate}
\end{definition}

A biCartesian square of the form \eqref{eq:biCartDiag} with $W=0$ is called a \emph{conflation} or \emph{short exact sequence} and is denoted by $U \rightarrowtail V \twoheadrightarrow X$.

Examples of proto-exact categories include abelian categories, where $\mathfrak{I}$ and $\mathfrak{D}$ are the monomorphisms and epimorphisms, respectively, and exact categories, where $\mathfrak{I}$ and $\mathfrak{D}$ are the admissible monomorphisms and admissible epimorphisms, respectively. A non-exact example is as follows.

\begin{example}[\cite{szczesny2012}]
\label{eq:F1Vect}
The \emph{category $\vect_{\F_1}$ of finite dimensional vector spaces over $\F_1$} is defined as follows. An object of $\vect_{\F_1}$ is a pair ($V,*)$ consisting of a finite set $V$ with a basepoint $* \in V$. A morphism $f: (V,*) \rightarrow (W,*)$ in $\vect_{\F_1}$ is a map of pointed sets for which the restriction of $f$ to $V \setminus f^{-1}(*)$ is injective. The category $\vect_{\F_1}$ admits a proto-exact structure in which inflations (resp. deflations) are injective (resp. surjective) morphisms. Note that $\vect_{\F_1}$ is not pre-additive and so, in particular, is not exact.
\end{example}

\subsection{The Waldhausen $\Wald_{\bullet}$-construction}
\label{sec:waldhausen}

The \emph{Waldhausen $\Wald_{\bullet}$-construction} of a proto-exact category $\CC$ is the simplicial groupoid $\Wald_{\bullet}(\CC)$ whose $n$-simplices $\Wald_n(\CC)$ consist of all diagrams in $\CC$ of the form
\[
\begin{tikzpicture}[baseline= (a).base]
\node[scale=1] (a) at (0,0){
\begin{tikzcd}
0 \arrow[tail]{r} & A_{\{0,1\}} \arrow[two heads]{d} \arrow[tail]{r} & \cdots \arrow[tail]{r} & A_{\{0,n-1\}} \arrow[two heads]{d}  \arrow[tail]{r} & A_{\{0,n\}} \arrow[two heads]{d} \\
 & 0 \arrow[tail]{r} & \cdots   \arrow[tail]{r} & A_{\{1,n-1\}} \arrow[two heads]{d}  \arrow[tail]{r} & A_{\{1,n\}} \arrow[two heads]{d} \\
 &  & \ddots & \vdots \arrow[two heads]{d} & \vdots \arrow[two heads]{d} \\
  &  & & 0  \arrow[tail]{r} & A_{\{n-1,n\}} \arrow[two heads]{d} \\
  &  &  &  & 0
\end{tikzcd}
};
\end{tikzpicture}
\]
such that each sub-rectangle is biCartesian.
%
%
An isomorphism $A_{\{\bullet,\bullet\}} \xrightarrow[]{\sim} B_{\{\bullet,\bullet\}}$ in $\Wald_n(\CC)$ is the data of isomorphisms $A_{\{i,j\}} \xrightarrow[]{\sim} B_{\{i,j\}}$, $0 \leq i \leq j \leq n$, which make the obvious diagrams commute. The degeneracy map $s_i: \Wald_n(\CC) \rightarrow \Wald_{n+1}(\CC)$ inserts a row and column of identity morphisms after the $i$\textsuperscript{th} row and $i$\textsuperscript{th} column and the face map $\partial_i: \Wald_n(\CC) \rightarrow \Wald_{n-1}(\CC)$ deletes the $i$\textsuperscript{th} row and $i$\textsuperscript{th} column and then composes the adjacent morphisms.

We make explicit the simplices in low degrees. The $0$-simplices $\Wald_0(\CC)=\{0\}$ form the trivial groupoid while $\Wald_1(\CC) \simeq \CC^{\sim}$
is the maximal groupoid of $\CC$. The $2$-simplices form the groupoid of conflations in $\CC$:
\begin{equation}\label{eq:walds2}
\Wald_2(\CC) =
\left\{
\begin{tikzpicture}[baseline= (a).base]
\node[scale=1] (a) at (0,0){
\begin{tikzcd}
0 \arrow[tail]{r} & A_{\{0,1\}} \arrow[two heads]{d} \arrow[tail]{r} & A_{\{0,2\}} \arrow[two heads]{d}\\
 & 0 \arrow[tail]{r} & A_{\{1,2\}} \arrow[two heads]{d} \\
 & & 0
\end{tikzcd}
};
\end{tikzpicture}
\right\}
\simeq
\{ U \rightarrowtail V \twoheadrightarrow W\}.
\end{equation}
With respect to these equivalences, the face maps $\partial_0, \partial_1, \partial_2: \Wald_2(\CC) \rightarrow \Wald_1(\CC)$ send a conflation $U \rightarrowtail V \twoheadrightarrow W$ to $W$, $V$ and $U$, respectively.

The simplicial groupoid $\Wald_{\bullet}(\CC)$ was originally defined in the context of algebraic $K$-theory \cite{waldhausen1985}, where for a Waldhausen category $\CC$---for example, an exact category---it was proved that there is an isomorphism of groups
\begin{equation}
\label{eq:WaldK}
K_i(\CC) \simeq \pi_{i+1} \vert \Wald_{\bullet}(\CC) \vert.
\qquad
 i \geq 0
\end{equation}
Here $\vert \Wald_{\bullet}(\CC) \vert$ denotes the geometric realization of the nerve of $\Wald_{\bullet}(\CC)$.

\begin{theorem}[{\cite[Proposition 2.4.8]{dyckerhoff2019}}]
\label{thm:Wald2Seg}
The simplicial groupoid $\Wald_{\bullet}(\CC)$ is $2$-Segal.
\end{theorem}

The proof of Theorem \ref{thm:Wald2Seg} is elementary, effectively reducing to the Third Isomorphism Theorem for $\CC$. The proof extends with only minor changes to the setting in which $\CC$ is an exact $\infty$-category; see \cite[Theorem 7.3.3]{dyckerhoff2019}.

\section{Hall algebras via $2$-Segal spaces}\label{twosec}
\label{sec:hallAlg2Seg}

Throughout this section $X_{\bullet}: \Delta^{\op} \rightarrow \spaces$ denotes a $2$-Segal space.

\subsection{An algebra in spans of groupoids}
\label{sec:algSpan}

In this section we take $\spaces$ to be the model category of small groupoids; this suffices for all explicit examples below. Modifications (and strengthenings of) the content of this section when $\spaces$ is an $\infty$-category can be found in \cite{stern2021}.

Let $\Span(\spaces)$ be the category of spans in $\spaces$. An object of $\Span(\spaces)$ is an object of $\spaces$. A morphism $A \rightarrow B$ in $\Span(\spaces)$ is an equivalence class of diagrams $A \leftarrow V \rightarrow B$ in $\spaces$, with two diagrams $A \leftarrow V \rightarrow B$ and $A \leftarrow V^{\prime} \rightarrow B$ being equivalent if there exists an equivalence $V \rightarrow V^{\prime}$ making the obvious diagrams commute. A monoidal structure on $\Span(\spaces)$ is induced by Cartesian product.

Let $m: X_1 \times_{X_0} X_1 \rightarrow X_1$ be the following morphism in $\Span(\spaces)$:
\begin{equation}
\label{eq:multSpan}
\begin{tikzpicture}[baseline= (a).base]
\node[scale=1.0] (a) at (0,0){
\begin{tikzcd}
{} & X_2 \arrow{dl}[above left]{(\partial_2,\partial_0)} \arrow{dr}[above right]{\partial_1} & {} \\
X_{\{0,1\}} \times_{X_{\{1\}}} X_{\{1,2\}} & {} & X_{\{0,1\}}.
\end{tikzcd}
};
\end{tikzpicture}
\end{equation}
The composition $m \circ (m \times \id_{X_1})$ corresponds to the outside span in the diagram
\begin{equation}
\label{diag:assSpanPre}
\begin{tikzpicture}[baseline= (a).base]
\node[scale=0.60] (a) at (0,0){
\begin{tikzcd}
{} & {} & K_{m \circ (m \times \id_{X_1})} \arrow{dl} \arrow{dr}  \arrow[d, phantom, "\usebox\pullbackvert",near start, color=black] & {} & {} \\
{} & X_{\{0,1,2\}} \times_{X_{\{2\}}} X_{\{2,3\}} \arrow{dl} \arrow{dr}& {} & X_{\{0,2,3\}} \arrow{dl} \arrow{dr} & {} \\
X_{\{0,1\}} \times_{X_{\{1\}}} X_{\{1,2\}} \times_{X_{\{2\}}} X_{\{2,3\}} & {} & X_{\{0,2\}} \times_{X_{\{2\}}} X_{\{2,3\}} & {} & X_{\{0,3\}}.
\end{tikzcd}
};
\end{tikzpicture}
\end{equation}
The carrot in the top corner indicates that the square is a pullback. The apex $K_{m \circ (m \times \id_{X_1})}$ is the middle term in the sequence
\[
X_3
\xrightarrow[]{f}
K_{m \circ (m \times \id_{X_1})}
\xrightarrow[]{g}
X_{\{0,1,2\}} \times_{X_{\{0,2\}}} X_{\{0,2,3\}},
\]
where $f$ is induced by the canonical maps $X_3 \rightarrow X_{\{0,1,2\}} \times_{X_{\{2\}}} X_{\{2,3\}}$ and $X_3 \rightarrow X_{\{0,2,3\}}$. The map $g$ is clearly an equivalence and the composition $g \circ f$ is an equivalence by the $2$-Segal property of $X_{\bullet}$; take $n=3$, $i=0$, $j=2$ in Definition \ref{def:2Segal}. The 2-out-of-3 property of equivalences therefore implies that $f$ is an equivalence. Hence, the diagram \eqref{diag:assSpanPre} is equivalent
\begin{equation}
\label{diag:assSpan}
\begin{tikzpicture}[baseline= (a).base]
\node[scale=0.60] (a) at (0,0){
\begin{tikzcd}
{} & {} & X_3 \arrow{dl}[above]{\gamma^{\prime}} \arrow{dr}[above]{\beta^{\prime}} \arrow[d, phantom, "\usebox\pullbackvert",near start, color=black] & {} & {} \\
{} & X_{\{0,1,2\}} \times_{X_{\{2\}}} X_{\{2,3\}} \arrow{dl}[above]{\alpha} \arrow{dr}[above]{\beta} & {} & X_{\{0,2,3\}} \arrow{dl}[above]{\gamma} \arrow{dr}[above]{\delta} & {} \\
X_{\{0,1\}} \times_{X_{\{1\}}} X_{\{1,2\}} \times_{X_{\{2\}}} X_{\{2,3\}} & {} & X_{\{0,2\}} \times_{X_{\{2\}}} X_{\{2,3\}} & {} & X_{\{0,3\}}.
\end{tikzcd}
};
\end{tikzpicture}
\end{equation}
It follows that $m \circ (m \times \id_{X_1})$ is equal to the span
\begin{equation}
\label{eq:assSpanFinal}
\begin{tikzpicture}[baseline= (a).base]
\node[scale=1.0] (a) at (0,0){
\begin{tikzcd}
{} & X_3 \arrow{dl}[above left]{\alpha \circ \gamma^{\prime}} \arrow{dr}[above right]{\delta \circ \beta^{\prime}} & {} \\
X_{\{0,1\}} \times_{X_{\{1\}}} X_{\{1,2\}} \times_{X_{\{2\}}} X_{\{2,3\}} & {} & X_{\{0,3\}}.
\end{tikzcd}
};
\end{tikzpicture}
\end{equation}
Similarly, the composition $m \circ (\id_{X_1} \times m)$ is given by the outside span in the diagram
\begin{equation}
\label{diag:assSpan2}
\begin{tikzpicture}[baseline= (a).base]
\node[scale=0.60] (a) at (0,0){
\begin{tikzcd}
{} & {} & K_{m \circ (\id_{X_1} \times m)} \arrow{dl} \arrow{dr} \arrow[d, phantom, "\usebox\pullbackvert",near start, color=black] & {} & {} \\
{} & X_{\{0,1\}} \times_{X_{\{1\}}} X_{\{1,2,3\}} \arrow{dl} \arrow{dr} & {} & X_{\{0,1,3\}} \arrow{dl} \arrow{dr} & {} \\
X_{\{0,1\}} \times_{X_{\{1\}}} X_{\{1,2\}} \times_{X_{\{2\}}} X_{\{2,3\}} & {} & X_{\{0,1\}} \times_{X_{\{1\}}} X_{\{1,3\}} & {} & X_{\{0,3\}}.
\end{tikzcd}
};
\end{tikzpicture}
\end{equation}
Arguing as above, this time using the case $n=3$, $i=1$, $j=3$ of Definition \ref{def:2Segal}, the outside span is again seen to equal \eqref{eq:assSpanFinal}.

Summarizing, the pair $(X_1,m)$ is an associative algebra object in $\Span(\spaces)$, associativity resulting from the first non-trivial $2$-Segal conditions, that is, the conditions corresponding to $n=3$ in Definition \ref{def:2Segal}.

%
%

\subsection{Linearization via theories with transfer}
\label{sec:trans}

Following \cite[\S 8.1]{dyckerhoff2019}, we construct algebras from a $2$-Segal space $X_{\bullet}$. When $\spaces$ is the category small groupoids, this can be seen as linearization of the algebra object $(X_1,m)$ in $\Span(\spaces)$ constructed in Section \ref{sec:algSpan}.

\begin{definition}
Subclasses $\sm$ and $\pr$ of morphisms in $\spaces$ are called \emph{smooth} and \emph{proper}, respectively, if the following properties hold:
\begin{enumerate}
\item Each subclass contains all weak equivalences and is closed under composition and Cartesian product. In particular, $\sm$ and $\pr$ may be seen as subcategories of $\spaces$ which contain all objects.
\item For each Cartesian diagram
\begin{equation}
\label{eq:BBSquare}
\begin{tikzpicture}[baseline= (a).base]
\node[scale=1.0] (a) at (0,0){
\begin{tikzcd}
{} & X \arrow{dl}[above left]{s^{\prime}} \arrow{dr}[above right]{p^{\prime}} & {} \\
Y \arrow{dr}[below left]{p} & {} & Z \arrow{dl}[below right]{s} \\
{} & W& {}
\end{tikzcd}
};
\end{tikzpicture}
\end{equation}
in $\CC$ with $s \in \sm$ and $p \in \pr$, we have $s^{\prime} \in \sm$ and $p^{\prime} \in \pr$.
\end{enumerate}
\end{definition}

Let $\mathcal{V}$ be a monoidal category with monoidal unit $\mathbb{I}_{\mathcal{V}}$.

\begin{definition}
A \emph{theory with transfer on $\spaces$ valued in $\mathcal{V}$} is the data of a covariant functor $(-)_*: \sm \rightarrow \mathcal{V}$ and a contravariant functor $(-)^*: \pr \rightarrow \mathcal{V}$ with the following properties:
\begin{enumerate}
\item Both $(-)_*$ and $(-)^*$ send weak equivalences in $\spaces$ to isomorphisms in $\mathcal{V}$.
\item $X_*=X^*$ for all objects $X \in \spaces$; denote this common object by $\lin(X)$.
\item Monoidal data: there are isomorphisms $\lin(\pt) \xrightarrow[]{\sim} \mathbb{I}_{\mathcal{V}}$ and
\[
\{m_{X,Y}: \lin(X) \otimes \lin(Y) \xrightarrow[]{\sim} \lin(X \times Y)\}_{X,Y \in \spaces}
\]
which are natural with respect to $(-)^*$ and $(-)_*$ and satisfy the standard associativity and unit coherence conditions.
\item Given a Cartesian diagram \eqref{eq:BBSquare} with $s \in \sm$ and $p \in \pr$, the Beck--Chevalley-type equality
\[
s^* \circ p_* = p^{\prime}_* \circ s^{\prime *}
\]
of morphisms $\lin(Y) \rightarrow \lin(Z)$ holds.
\end{enumerate}
\end{definition}

For ease of notation, we denote a theory with transfer simply by $\lin$.

Recall again the span \eqref{eq:multSpan} which defines the multiplication $m$ on $X_1$.
Consider also the span
\begin{equation}
\label{eq:unitSpan}
\begin{tikzpicture}[baseline= (a).base]
\node[scale=1.0] (a) at (0,0){
\begin{tikzcd}
{} & X_0 \arrow{dl}[above left]{\can} \arrow{dr}[above right]{s_0} & {} \\
\pt & {} & X_1,
\end{tikzcd}
};
\end{tikzpicture}
\end{equation}
where $\can$ is the canonical map.

\begin{theorem}[{\cite[Proposition 8.1.5]{dyckerhoff2019}}]
\label{thm:linearize}
Let $\lin$ be a theory with transfer on $\spaces$ valued in $\mathcal{V}$ and $X_{\bullet}: \Delta^{\op} \rightarrow \spaces$ a $2$-Segal space. If $(\partial_2,\partial_0) \in \sm$ and $\partial_1 \in \pr$, then $\Hall(X_{\bullet}; \lin) := \lin(X_1)$ is an associative algebra object in $\mathcal{V}$ with multiplication
\[
\lin(X_1) \otimes \lin(X_1) \simeq \lin(X_1 \times X_1) \xrightarrow[]{\partial_{1*} \circ (\partial_2,\partial_0)^*} \lin(X_1)
\]
induced by the span \eqref{eq:multSpan}. If, moreover, $\can \in \sm$ and $s_0 \in \pr$, then the morphism
\[
\mathbb{I}_{\mathcal{V}} \xrightarrow[]{\sim} \lin(\pt) \xrightarrow[]{s_{0 *} \circ \can^*} \lin(X_1)
\]
induced by the span \eqref{eq:unitSpan} is a unit for the algebra $\Hall(X_{\bullet}; \lin)$.
\end{theorem}

\begin{proof}
Write $\mathfrak{m}$ for the multiplication $\lin(X_1) \otimes \lin(X_1) \rightarrow \lin(X_1)$. Keeping the notation of diagram \eqref{diag:assSpan} and using the monoidality coherence data, the composition $\mathfrak{m} \circ (\mathfrak{m} \otimes \id_{\lin(X_1)})$ is equal to
\[
\delta_* \circ \gamma^* \circ \beta_* \circ \alpha^*
=
\delta_* \circ \beta^{\prime}_* \circ \gamma^{\prime *} \circ \alpha^*
=
(\delta \circ \beta^{\prime})_* \circ (\alpha \circ \gamma^{\prime})^*.
\]
The first equality is the Beck--Chevalley property and the second is functoriality of $(-)_*$ and $(-)^*$. The final expression is the result of push-pull along the span \eqref{eq:assSpanFinal}. Similarly, the composition $\mathfrak{m} \circ ( \id_{\lin(X_1)} \otimes \mathfrak{m})$, which results from push-pull along the span \eqref{diag:assSpan2}, is equal to the composition $(\delta \circ \beta^{\prime})_* \circ (\alpha \circ \gamma^{\prime})^*$. This establishes the associativity of $\mathfrak{m}$.

By Theorem \ref{thm:unital}, the $2$-Segal space $X_{\bullet}$ is unital. Using this, it is straightforward to verify unitality of $\lin(X_1)$.
\end{proof}

\subsection{Hall algebras of finitary categories}
\label{sec:finCat}

We specialize the results of Section \ref{sec:trans} to recover the theory of Hall algebras, as formulated by Ringel \cite{ringel1990}. 

Given a small groupoid $\mathcal{A}$, let $\lin(\mathcal{A}) = \Fun_c(\pi_0 \mathcal{A}, \C)$ be the vector space of finitely supported complex valued functions on the set $\pi_0 \mathcal{A}$ of isomorphism classes of objects of $\mathcal{A}$. Let $\sm$ consist of the weakly proper morphisms, that is, morphisms of groupoids $f: \mathcal{A} \rightarrow \mathcal{B}$ for which the induced map $\pi_0 f : \pi_0 \mathcal{A} \rightarrow \pi_0 \mathcal{B}$ is finite-to-one. In this case, there is an obvious pullback morphism $f^* : \lin(\mathcal{B}) \rightarrow \lin(\mathcal{A})$.  Let $\pr$ consist of the locally proper morphisms, that is, morphisms of groupoids $f: \mathcal{A} \rightarrow \mathcal{B}$ for which, for each $a \in \mathcal{A}$, the induced group homomorphism $f: \Aut_{\mathcal{A}}(a) \rightarrow \Aut_{\mathcal{B}}(f(a))$ has finite kernel and cokernel. In this case, there is a pushforward $f_*: \lin(\mathcal{A}) \rightarrow \lin(\mathcal{B})$ given by
\[
(f_* \phi)(b) = \int_{Rf^{-1}(b)} \phi_{\vert Rf^{-1}(b)},
\]
where $Rf^{-1}(b)$ denotes the homotopy fibre of $f$ over $b \in \mathcal{B}$. Here, for a small groupoid $\mathcal{G}$ for which each group $\Aut_{\mathcal{G}}(x)$, $x \in \mathcal{G}$, is finite with function $\phi \in \lin(\mathcal{G})$, we have written
\[
\int_{\mathcal{G}} \phi = \sum_{x \in \pi_0 \mathcal{G}} \frac{1}{\vert \Aut_{\mathcal{G}} (x) \vert} \phi(x).
\]
This data defines a $\Vect_{\C}$-valued theory with transfer on the model category of small groupoids \cite[Lemma 8.2.7]{dyckerhoff2019}.

\begin{definition}
A proto-exact category $\CC$ is called \emph{finitary} if it is essentially small and, for each $U,V \in \CC$, the sets $\Hom_{\CC}(U,V)$ and $\Ext^1_{\CC}(U,V)$ are finite.
\end{definition}

Let $X_{\bullet} = \Wald_{\bullet}(\CC)$ be the Waldhausen construction of a finitary proto-exact category $\CC$, seen as a $2$-Segal simplicial groupoid. The span \eqref{eq:multSpan} takes the form
\[
\begin{tikzpicture}[baseline= (a).base]
\node[scale=1.0] (a) at (0,0){
\begin{tikzcd}[,column sep=0.25in]
{} & \Wald_2(\CC) \arrow{dl}[above left]{(\partial_2,\partial_0)} \arrow{dr}[above right]{\partial_1}& {} \\
\Wald_1(\CC) \times \Wald_1(\CC) & {} & \Wald_1(\CC)
\end{tikzcd}
};
\end{tikzpicture},
\qquad
\begin{tikzpicture}[baseline= (a).base]
\node[scale=1.0] (a) at (0,0){
\begin{tikzcd}[arrows={|->},column sep=0.2in]
{} & U \rightarrowtail V \twoheadrightarrow W \arrow{dl} \arrow{dr} & {} \\
(U,W) & {} & V.
\end{tikzcd}
};
\end{tikzpicture}
\]

\begin{lemma}
\label{lem:finImpliesSmPr}
\phantomsection
\begin{enumerate}
\item The morphism $(\partial_2,\partial_0)$ is weakly proper, that is, $(\partial_2,\partial_0) \in \sm$.
\item The morphism $\partial_1$ is locally proper, that is, $\partial_1 \in \pr$.
\end{enumerate}
\end{lemma}

\begin{proof}
\begin{enumerate}
\item This follows from the assumption that each set $\Hom_{\CC}(W,U)$ is finite. 
\item This follows from the assumption that each set $\Ext^1_{\CC}(W,U)$ is finite. \qedhere
\end{enumerate}
\end{proof}

It follows from Theorem \ref{thm:linearize} that there is a unital associative algebra
\[
\Hall(\CC) := \Hall(\Wald_{\bullet}(\CC); \lin).
\]
The algebra $\Hall(\CC)$ is the \emph{Hall algebra of $\CC$}, originally defined by Ringel \cite{ringel1990}. The vector space $\Hall(\CC)$ has a canonical basis $\{1_U\}_{U \in \pi_0 \CC}$ labelled by isomorphism classes of objects of $\CC$ in which the multiplication reads
\[
1_U \cdot 1_W = \sum_{V \in \pi_0 \CC} c_{U,W}^V 1_V.
\]
Here $c_{U,W}^V$ is the cardinality of the set 
\[
\{U^{\prime} \subset V \mid U^{\prime} \simeq U, \; V \slash U^{\prime} \simeq W \}.
\]
Alternatively, we can write
\[
c_{U,W}^V
=
\frac{\vert \Aut_{\CC}(V) \vert} {\vert \Aut_{\CC}(U) \vert \vert \Aut_{\CC}(W) \vert} \frac{\vert \Ext^1_{\CC}(W,U)_V \vert}{\vert \Hom_{\CC}(W,U) \vert},
\]
where $\Ext^1_{\CC}(W,U)_V \subset \Ext^1_{\CC}(W,U)$ denotes the subset of extensions (seen as isomorphism classes of conflations) whose middle term is isomorphic to $V$. In this form, it is apparent that the finitary assumption on $\CC$ ensures that $c_{U,W}^V$ is finite and, for fixed $U,W \in \pi_0 \CC$, is non-zero for only finitely many $V \in \pi_0 \CC$.

We work out various instances of the above construction. Let $\F_q$ be a finite field of cardinality $q$.

\begin{example}
\label{ex:pointFq}
Let $\vect_{\F_q}$ be the finitary abelian category of finite dimensional vector spaces over $\F_q$. The Hall algebra
\[
\Hall(\vect_{\F_q}) \simeq \bigoplus_{n \in \Z_{\geq 0}} \C \cdot 1_n
\]
has multiplication determined by
\[
1_n \cdot 1_m = \left[ \begin{matrix} n+m \\ n \end{matrix} \right]_q 1_{n+m},
\]
where our conventions for quantum integers read
\[
[n]_q = \frac{q^n-1}{q-1},
\qquad
[n]_q! = \prod_{i=1}^n [i]_q
\qquad
\left[ \begin{matrix} n+m \\ n \end{matrix} \right]_q
=
\frac{[n+m]_q!}{[n]_q! [m]_q!}.
\]
The structure constant $\left[ \begin{smallmatrix} n+m \\ n \end{smallmatrix} \right]_q$ is the number of $\F_q$-rational points of the Grassmannian $\textnormal{Gr}(n, n+m)$.
\end{example}
 
\begin{example}
\label{eq:Jordan}
Let $\F_q[x] \nilMod$ be the finitary abelian category of finite dimensional nilpotent $\F_q[x]$-modules. Jordan canonical form induces a bijection from the set of isomorphism classes of objects of $\F_q[x] \nilMod$ to the set $\mathcal{P}$ of partitions, so that we may identify
\[
\Hall(\F_q[x] \nilMod) \simeq \bigoplus_{\lambda \in \mathcal{P}} \C \cdot 1_{\lambda}.
\]
The structure constants $c_{\lambda,\mu}^{\nu}$ are rather complicated and do not admit a closed form expression. However, studying these coefficients when $\mu$ is of the form $(1^r)$ allows one to prove that $\Hall(\F_q[x] \nilMod)$ is a polynomial algebra on the countably infinite variables $1_{(1^r)}$, $r \geq 1$.
Using this, it is not difficult to show that $\Hall(\F_q[x] \nilMod)$ is isomorphic to Macdonald's ring of symmetric functions \cite{macdonald1995}, an important algebra in combinatorics. 

Taking $q$ to be a prime $p$, the algebra $\Hall(\F_p[x] \nilMod)$ is isomorphic to the original Hall algebra, introduced by Steinitz in the context of the combinatorics of the lattice of finite $p$-abelian groups  \cite{steinitz1901}.
\end{example}

\begin{example}[{\cite{ringel1990b,green1995}}]
\label{ex:quiverFq}
Let $Q$ be an acyclic quiver, that is, a finite oriented multigraph without cycles. Let $\rep_{\F_q}(Q)$ be the category of finite dimensional representations of $Q$ over $\F_q$. Concretely, a representation of $Q$ is the data of a finite dimensional vector space $V_i$ over $\F_q$ for each vertex $i$ of $Q$ and a $\F_q$-linear map $V_i \rightarrow V_j$ for each arrow $i \rightarrow j$ of $Q$. The category $\CC$ is finitary abelian. Denote by $S_i$ the one dimensional representation which is supported only at the vertex $i$. Any simple object of $\CC$ is isomorphic to a unique $S_i$.

Let $\mathfrak{g}_Q$ be the complex generalized Ka\u{c}--Moody algebra whose generalized Dynkin diagram is the underlying multigraph of $Q$. There is a bijection between the vertices of $Q$ and the set $\Pi^+$ of positive simple roots of $\mathfrak{g}_Q$. Let $U_v(\mathfrak{g}_Q)$ be the De Concini--Ka\v{c} quantum group at quantum parameter $v \in \C^{\times}$. The algebra $U_v(\mathfrak{g}_Q)$ has a standard presentation by generators and relations; Chevalley generators $E_i, F_i$, $i \in \Pi^+$, and group-like Cartan generators $K^{\pm 1}_{\alpha}$, $\alpha$ a simple coroot. Let $U_v^+(\mathfrak{g}_Q)$ be the subalgebra of $U_v(\mathfrak{g}_Q)$ generated by $E_i$; its only relations are the quantum Serre relations.

It was proved by Ringel and Green that the assignment $E_i \mapsto 1_{S_i}$ extends to an algebra embedding
\begin{equation*}
U_{\sqrt{q}}^+(\mathfrak{g}_Q) \hookrightarrow \Hall(\rep_{\F_q}(Q))
\end{equation*}
which is an isomorphism if and only if $Q$ is of finite representation type, that is, $Q$ is an orientation of a simply laced Dynkin diagram \cite{ringel1990b,green1995}. The algebra $H(\rep_{\F_q}(Q))$ itself is in fact a quantization of a $q$-dependent Borcherds algebra \cite{sevenhant2001}. For a general discussion of the relationship between Hall algebras of finitary exact categories and quantum groups, see \cite{berenstein2016}.

More generally, if $Q$ has cycles (including loops), one can consider the category of representations of $Q$ with nilpotency conditions on morphisms associated to cycles. For example, the category $\rep^{\nil}_{\F_q}(Q_{\Jor})$ of nilpotent representations of the Jordan quiver $Q_{\Jor}$ (a single vertex with a loop) is equivalent to the category $\F_q[x] \nilMod$ considered in Example \ref{eq:Jordan}. In this way, one can expand the class of generalized quantum Ka\u{c}--Moody algebras realized by Hall algebras \cite{kang2006}. 

The theory of Hall algebras of quivers is rich and interesting. A detailed introduction can be found in \cite{schiffmann2012b}.
\end{example}

\begin{example}[\cite{kapranov1997}]
\label{ex:curveFq}
Let $X$ be a smooth projective irreducible curve over a finite field $\F_q$. The category $\coh(X)$ of coherent sheaves on $X$ is finitary abelian and hence has an associated Hall algebra $H(X) := H(\coh(X))$.

The case of $X=\mathbb{P}^1$, the projective line over $\F_q$, was studied by Kapranov \cite{kapranov1997}. See \cite{baumann2001} for an expository account. To explain the main result, note that the full subcategory $\tor(\mathbb{P}^1) \subset \coh(\mathbb{P}^1)$ of torsion sheaves is abelian and closed under extensions. Using functoriality of the Hall algebra construction, as discussed in Section \ref{sec:functoriality}, this translates to define a subalgebra $H(\tor(\mathbb{P}^1))$ of $H(\mathbb{P}^1)$. Consideration of the support of a torsion sheaf leads to a decomposition
\[
\tor(\mathbb{P}^1) = \bigoplus_{x \in \mathbb{P}^1} \tor(\mathbb{P}^1)_x
\]
labelled by the closed points of $\mathbb{P}^1$ and hence an isomorphism of algebras
\[
H(\tor(\mathbb{P}^1)) \simeq \bigotimes_{x \in \mathbb{P}^1} H(\tor(\mathbb{P}^1)_x).
\]
The abelian equivalence $\tor(\mathbb{P}^1)_x \simeq \rep^{\nil}_{\F_{q^{\deg x}}}(Q_{\Jor})$ yields, via Example \ref{eq:Jordan}, a description of $H(\tor(\mathbb{P}^1))$ in terms of algebras of symmetric functions. Similarly, the subcategory $\vect(\mathbb{P}^1) \subset \coh(\mathbb{P}^1)$ of vector bundles and their morphisms of constant rank is exact and closed under extensions. The resulting subalgebra $H(\vect(\mathbb{P}^1))$ of $H(\mathbb{P}^1)$ can be interpreted as the algebra of unramified automorphic forms for the groups $GL_n$, $n \geq 0$, with parabolic induction as the product. Finally, that the pair $(\tor(\mathbb{P}^1),\vect(\mathbb{P}^1))$ is a torsion theory for $\coh(\mathbb{P}^1)$ translates to the statement that $\Hall(\mathbb{P}^1)$ is a semi-direct product of $H(\tor(\mathbb{P}^1))$ and $H(\vect(\mathbb{P}^1))$ with respect to the Hecke action of the former on the latter. By taking a subalgebra $\tilde{H}(\tor(\mathbb{P}^1))$ of $H(\tor(\mathbb{P}^1))$ obtained as a suitable average of $H(\tor(\mathbb{P}^1))_x$, $x \in \mathbb{P}^1$, it is proved that there is an algebra isomorphism
\[
\tilde{H}(\tor(\mathbb{P}^1)) \rtimes H(\vect(\mathbb{P}^1))
\simeq
U_{\sqrt{q}}^{D,+}(\widehat{\mathfrak{sl}_2(\C)})
\]
and hence a non-surjective algebra embedding
\[
U_{\sqrt{q}}^{D,+}(\widehat{\mathfrak{sl}_2(\C)}) \hookrightarrow \Hall(\mathbb{P}^1).
\]
Here $U_{\sqrt{q}}^{D,+}(\widehat{\mathfrak{sl}_2(\C)})$ denotes the positive part of the quantum affine algebra $U_v(\widehat{\mathfrak{sl}_2(\C)})$ in Drinfeld's new realization \cite{drinfeld1987}, specialized to $v = \sqrt{q}$. This  realization is particularly natural when $\widehat{\mathfrak{sl}_2(\C)}$ is constructed from the loop algebra of $\mathfrak{sl}_2(\C)$.

On the other hand, denoting by $K_2$ the Kronecker quiver (vertex set $\{1,2\}$ with two arrows from $1$ to $2$), Example \ref{ex:quiverFq} asserts that $\Hall(\rep_{\F_q}(K_2))$ contains as a subalgebra the positive part of the standard Chevalley presentation of $U_{\sqrt{q}}(\widehat{\mathfrak{sl}_2(\C)})$. To explain this connection, note that while the abelian categories $\coh(\mathbb{P}^1_{\slash \F_q})$ and $\rep_{\F_q}(K_2)$ are not equivalent---the latter is finite length, the former is not---Beilinson proved that there is an equivalence of their derived categories \cite{beilinson1978}:
\begin{equation}
\label{eq:beilinsonEquiv}
D^b(\coh(\mathbb{P}^1_{\slash \F_q})) \simeq D^b(\rep_{\F_q}(K_2)).
\end{equation}
This leads to the idea that there should exist a derived generalization of the Hall algebra construction which, when applied to the category \eqref{eq:beilinsonEquiv}, realizes the full quantum group $U_{\sqrt{q}}(\widehat{\mathfrak{sl}_2(\C)})$, whereas the hearts $\rep_{\F_q}(K_2)$ and $\coh(\mathbb{P}^1_{\slash \F_q})$ realize only (non-conjugate) positive parts thereof. See \cite{peng2000,cramer2010,bridgeland2013} for work in this general direction.

Hall algebras of elliptic curves have also been studied in detail; see \cite{burban2012,schiffmann2012c}. Hall algebras of higher genus curves remain largely mysterious; see \cite{kapranov2017}.
\end{example}

\begin{remark}
In many of the examples above, the Hall algebras arise from $\F_q$-reductions of a category $\CC_{\Z}$ defined over $\Z$. What is often called the Hall algebra is an algebra $A$ over the Laurent polynomial ring which interpolates between the reduction:
$$\Hall(\CC_{\Z} \otimes_{\Z}\F_q) \simeq A\otimes_{\Z[q,q^{-1}]}\C.$$ 
One important subtlety in the subject is that the existence of such an algebra $A$ is not guaranteed.
\end{remark}

We turn to examples of Hall algebras arising from non-exact categories.

\begin{example}\label{ex:pointF1}
The proto-exact category $\vect_{\F_1}$ of finite dimensional vector spaces over $\F_1$, as defined in Example \ref{eq:F1Vect}, is finitary. The Hall algebra
\[
\Hall(\vect_{\F_1}) \simeq \bigoplus_{n \in \Z_{\geq 0}} \C \cdot 1_n
\]
has multiplication
\[
1_n \cdot 1_m = \left( \begin{matrix} n+m \\ n \end{matrix} \right) 1_{n+m}.
\]
This algebra is sometimes called the {\em divided power algebra} because of the identification with the integral subring $\mathbb{Z}\langle\frac{x^n}{n!} : n\geq 0\rangle \subset \C[x]$. The structure constant $\left( \begin{smallmatrix} n+m \\ n \end{smallmatrix} \right)$ is the Euler characteristic of the complex Grassmannian $\mbox{Gr}(n,n+m)$. In particular, this example is a $q \rightarrow 1$ limit of Example \ref{ex:pointFq} in a rather obvious sense.
\end{example}

\begin{example}[\cite{szczesny2012}]
\label{ex:quiverF1}
Let $\rep_{\F_1}(Q)$ be the category of representations of an acyclic quiver $Q$ in $\vect_{\F_1}$. The category $\CC$ inherits a proto-exact structure from $\vect_{\F_1}$ and is finitary. The Hall algebra $H(\rep_{\F_1}(Q))$ was studied in detail by Szczesny, who proved that the assignment $E_i \mapsto 1_{S_i}$ extends to an algebra homomorphism
\begin{equation}
\label{eq:uniEnvelMap}
U^+ (\mathfrak{g}_Q) \rightarrow \Hall(\rep_{\F_1}(Q)).
\end{equation}
Here $U^+ (\mathfrak{g}_Q)$ denotes the positive part of the (classical) enveloping algebra of $\mathfrak{g}_Q$. In general, the map \eqref{eq:uniEnvelMap} is neither surjective nor injective, even for Dynkin quivers, in contrast to Example \ref{ex:quiverFq}. To see the basic problem, let $Q$ be the inward orientation of the Dynkin diagram of type $D_4$, so that $\mathfrak{g}_Q \simeq \mathfrak{so}_8(\C)$. When $q$ is a power of a positive prime, the isomorphism class of the indecomposable representation
\[
\begin{tikzpicture}[scale=0.7]
    \node (A) at (0,0) {$\F_q^2$};
    \node (B) at (0,2.5) {$\F_q$};
    \node (C) at (-2.5,0) {$\F_q$};
    \node (D) at (2.5,0) {$\F_q$};
    \path [->] (B) edge node[left] {$\left(\begin{smallmatrix} 1 \\ 1 \end{smallmatrix}\right)$} (A);
    \path [->](C) edge node[above] {$\left(\begin{smallmatrix} 1 \\ 0 \end{smallmatrix}\right)$} (A);
    \path [->](D) edge node[above] {$\left(\begin{smallmatrix} 0 \\ 1 \end{smallmatrix}\right)$} (A);
\end{tikzpicture}
\]
realizes a non-simple positive root of $U_{\sqrt{q}}^+(\mathfrak{so}_8(\C))$. However, observe that this representation is not defined over $\F_1$: the matrix $\left(\begin{smallmatrix} 1 \\ 1 \end{smallmatrix}\right)$ does not represent a morphism in $\vect_{\F_1}$. In this example, non-surjectivity of the map \eqref{eq:uniEnvelMap} can be traced back to this observation. On the other hand, if $Q$ is an orientation of a Dynkin quiver of type $A$, where all indecomposable representations can be realized using only identity and zero matrices, the map \eqref{eq:uniEnvelMap} is an isomorphism. This illustrates the difficulty in establishing a precise relationship between $\Hall(\rep_{\F_1}(Q))$ and a $q \rightarrow 1$ limit of $\Hall(\rep_{\F_q}(Q))$.
\end{example}

\begin{example}[\cite{szczesny2012}]
\label{ex:curveF1}
Let $\F_1[t]$ be the free unital monoid with $0$ generated by $t$ and $\F_1[t,t^{-1}]$ its localization at the origin. The affine line $\mathbb{A}^1$, seen as a monoid scheme, is $\Spec \F_1[t]$. Then $\mathbb{P}^1_{\slash \F_1}$, the \emph{projective line over $\F_1$}, is the pushout of the diagram obtained by applying $\Spec$ to the diagram
\[
\begin{tikzpicture}[baseline= (a).base]
\node[scale=1] (a) at (0,0){
\begin{tikzcd}
& \F_1[t] \arrow[hook]{d}\\
\F_1[t] \arrow[hook]{r} & \F_1[t,t^{-1}].
\end{tikzcd}
};
\end{tikzpicture}
\]

The category $\coh(\mathbb{P}^1_{\slash \F_1})$ of (normal) coherent sheaves on $\mathbb{P}^1_{\slash \F_1}$ is proto-exact and finitary. Arguing much as in the case of $\mathbb{P}^1_{\slash \F_q}$, Szczesny proved that $H(\coh(\mathbb{P}^1_{\slash \F_1}))$ contains as a subalgebra a non-standard Borel subalgebra of $U(\mathcal{L} \mathfrak{sl}_2(\C))$. In fact, there is an isomorphism
\[
U(\mathcal{L} \mathfrak{gl}_2(\C)^+ \oplus \kappa)
\simeq
H(\coh(\mathbb{P}^1_{\slash \F_1})),
\]
where $\kappa$ is the abelian Lie algebra on generators $c_n$, $n \in \Z_{\geq 0}$. Geometrically, $c_n$ corresponds to the cyclic sheaf which globalizes the $\F_1[t]$-module $C_n=\{ \{z_i\}_{i \in \Z \slash n \Z},*\}$ on which $t$ acts by $t \cdot z_i = z_{i+1}$ and $0$ sends everything to $*$.

Elliptic Hall algebras over $\F_1$ were studied by Yanagida \cite{yanagida2017}, who proved that they recover the specialized elliptic Hall algebras of Morton--Samuelson \cite{morton2017} and hence skein algebras of tori. Over $\F_1$, the curve case stops here: curves of genus greater than one do not have monoid scheme models.

In a different direction, Hall algebras of toric monoid schemes were studied in \cite{szczesny2018,jun2024}. This is one approach to the largely unexplored theory of Hall algebras of higher dimensional varieties.
\end{example}

\begin{example}[\cite{eppolito2020}]
Let $\Mat$ be the category of pointed matroids and their strong maps. Given a matroid $M$, denote by $E_M$ the union of the non-zero elements of its ground set with a marked point $*$. Given a subset $S \subset E_M$, write $M_{\vert S}$ and $M \slash S$ for the restriction and contraction of $M$ along $S$, respectively. The category $\Mat$ admits a finitary proto-exact structure in which inflations and deflations are strong maps $f: N \rightarrow M$ which factor as $N \xrightarrow[]{\sim} M_{\vert S} \hookrightarrow M$ and $N \twoheadrightarrow N \slash S \xrightarrow[]{\sim} M$ for some subset $S$ of $E_M$ and $E_N$, respectively. The Hall algebra $H(\Mat)$ is a (Hopf) algebra dual of Schmitt's matroid-minor Hopf algebra \cite{schmitt1994}, which is of central importance in algebraic combinatorics. This algebra is not obviously related to a quasi-classical limit of some other algebra, in contrast to Examples \ref{ex:pointF1}-\ref{ex:curveF1}.

Relatedly, the proto-exact structure on $\Mat$ allows for the definition of its algebraic $K$-theory $K_{\bullet}(\Mat)$ via the isomorphism \eqref{eq:WaldK}. It is proved that this $K$-theory naturally contains as a subgroup the stable homotopy groups $\pi_{\bullet}^s(\mathbb{S})$ of the sphere spectrum.
\end{example}

We end this section with another choice of theory with transfer $\lin$.

\begin{example}
\label{ex:greenZele}
Let $\lin(-) = K_0([-,\vect_{\C}]_c)$ be the functor which assigns to an essentially small groupoid the Grothendieck group of its category of finitely supported complex representations; see \cite[\S 8.3]{dyckerhoff2019} for a full description of $\lin$ as a theory with transfer. The Hall algebra $H(\Wald_{\bullet}(\vect_{\F_q});\lin)$ is isomorphic to
\[
\bigoplus_{n=0}^{\infty} K_0(\rep_{\C}(\textnormal{GL}_n(\F_q)))
\]
with parabolic induction as the product. This algebra was studied in detail by Green \cite{green1955} and Zelevinsky \cite{zelevinsky1981}, who proved that it is a polynomial algebra on the cuspidal representations of the tower of groups $\textnormal{GL}_{n}(\F_q)$, $n \geq 0$.
\end{example}

\subsection{A non-finitary example: the cohomological Hall algebra}
\label{sec:CoHA}

Consider the non-finitary abelian category $\rep_{\C}(Q)$ of finite dimensional complex representations of a quiver $Q$. We consider $\Wald_{\bullet}(\rep_{\C}(Q))$ as a simplicial locally finite Artin stack.

Let $\ell(-) = H^{\bullet}(-;\Q)$ be singular cohomology. Applied to global quotient Artin stacks, $H^{\bullet}(-;\Q)$ is equivariant cohomology. The resulting Hall algebra $H(Q) := H(\Wald_{\bullet}(\rep_{\C}(Q));\lin)$, called the \emph{cohomological Hall algebra} of $Q$, was introduced by Kontsevich and Soibelman as a foundational object in Donaldson--Thomas theory \cite{kontsevich2011}, that is, the virtual counting of stable objects in $3$-dimensional Calabi--Yau categories. To be concrete, consider the case of the quiver $Q_m$ with a single node and $m \geq 0$ loops. In this case, there is an isomorphism
\[
\Wald_1(\rep_{\C}(Q_m))
\simeq
\bigsqcup_{d=0}^{\infty} [\mathfrak{gl}_d(\C)^{\oplus m} \slash \textnormal{GL}_d(\C)]
\]
with the stack of $m$-tuples of square matrices under simultaneous conjugation. The underlying $\Z \times \Z$-graded vector space of $H(Q_m)$ is
\[
H(Q_m) = \bigoplus_{d=0}^{\infty} H_d(Q_m)
\]
where
\[
H_d(Q_m) \simeq H^{\bullet}(B \textnormal{GL}_d(\C)) \simeq \Q[x_1,\dots,x_d]^{\mathfrak{S}_d},
\qquad
\vert x_i \vert = 2.
\]
The ($m$-dependent product) of $f_1 \in H_{d^{\prime}}(Q_m)$ and $f_2 \in H_{d^{\prime \prime}}(Q_m)$ reads
\[
f_1 \cdot f_2 = \sum_{\pi \in \mathfrak{sh}_{d^{\prime}, d^{\prime \prime}}} \pi \Big( f_1(x^{\prime}_1, \dots, x^{\prime}_{d^{\prime}})  f_2(x^{\prime \prime}_1, \dots,  x^{\prime \prime}_{d^{\prime \prime}})  \prod_{l=1}^{d^{\prime \prime}}  \prod_{k=1}^{d^{\prime}} ( x^{\prime \prime}_l - x^{\prime}_k )^{m-1} \Big),
\]
where $\mathfrak{sh}_{d^{\prime}, d^{\prime \prime}} \subset \mathfrak{S}_{d^{\prime} + d^{\prime \prime}}$ is the subset of the symmetric group on $d^{\prime}+d$ letters consists of shuffles of type $(d^{\prime}, d^{\prime \prime})$. By a direct study of this multiplication, Efimov proved that there exists a $\Z \times \Z$-graded vector space of the form $V_{Q_m}= V^{\prim}_{Q_m} \otimes \Q[u]$ and a graded algebra isomorphism\footnote{We are ignoring a dimension dependent shift (involving the Euler form) of the cohomological grading of $H(Q_m)$ which is required for this graded algebra isomorphism.} $\Sym(V_{Q_m}) \simeq H(Q_m)$ \cite{efimov2012}. Moreover, Efimov proved that each homogeneous summand $V^{\prim}_{Q_m,d} \subset V^{\prim}_{Q_m}$ is finite dimensional. The vector space $V^{\prim}_{Q_m}$ is the \emph{cohomological Donaldson--Thomas invariant} of $Q_m$. Its graded finite dimensionality was expected in view of its conjectured geometric interpretation. Indeed, it was later proved that $V^{\prim}_{Q_m}$ encodes the intersection cohomology groups of the closure of the moduli space of stable representations of $Q_m$ in the moduli space of semistable representations \cite{franzen2018,meinhardt2019}. In this way, structure theorems about the Hall algebra of a quiver translate to geometric properties of its moduli space of representations. These structural and geometric results generalize to symmetric quivers $Q$ with potential $W$ via \emph{critical} cohomological Hall algebras \cite{kontsevich2011,davison2017,davison2020}, where $V^{\prim}_{(Q,W)}$ is also given an interpretation of the long-sought Lie algebra of BPS states \cite{harvey1998}.

The linearization $\lin$ can be varied in a number of ways. For example, taking $\ell$ to be so-called motivic functions yields the \emph{motivic Hall algebra} of $Q$ \cite{joyce2007,bridgeland2012}. This is an extremely large algebra---much larger than the cohomological Hall algebra---which encodes the motivic Donaldson--Thomas invariants of $Q$ \cite{joyce2012} and has recently found connections to vertex operator algebras \cite{joyce2018}; see \cite{latyntsev2021} for similar connections in the cohomological setting.

\section{Functoriality between Hall algebras}
\label{sec:functoriality}

One advantage of using $2$-Segal spaces as a framework for understanding Hall
algebras is that clearer statements concerning functorality of the Hall algebra construction can be
introduced.  The purpose of this section is to survey materials which relate
to questions of functoriality, providing some details and highlights which
may be missing, but otherwise sourcing existent literature. 
The results in this section were understood by some of the experts at the conference which gave rise to this collection of surveys. This section was added because the statements were not contained the literature at the time.
A summary of statements pertaining to Hall (co)algebras is found in Theorem \ref{FunSummary} below.

Although we have adapted notations to agree with the rest of this document, the technical background for the exposition in this section largely follows from \cite{galvez2018}. In particular, we work with the skeleton of the simplicial category $\Delta$ and simplicial objects $X_\bullet : \Delta^{\op}\to \spaces$ take values in the large $\infty$-category of $\infty$-groupoids (as modeled by Kan complexes). 

Every map $f : [k]\to[n]$ in $\Delta$ admits a factorization $f=\alpha\circ \iota$ into an active map $\alpha$ and an inert map $\iota$. {\em
  Active maps} $\alpha : [k]\to [n]$ are those which satisfy $\alpha(0) = 0$
and $\alpha(k) = n$; they are generated by the codegeneracies $s^i :
[n+1]\to [n]$ and the inner coface maps $d^i :[n-1]\to [n]$, $0< i<
n$. {\em Inert maps} $\iota : [m]\to [k]$ are those which preserve distances
in the sense that $\iota(j+1) = \iota(j)+1$ for $0\leq j < m$.


Recall that a map $F : Y_\bullet \to X_\bullet$ between simplicial objects is {\em Cartesian} with respect to a map $f : [m]\to [n]$ in $\Delta$ if the diagram below is Cartesian.
\[
\begin{tikzpicture}[baseline= (a).base]
\node[scale=1.0] (a) at (0,0){
\begin{tikzcd}
Y_{n} \arrow{d}[left]{F_n} \arrow{r}[above]{f} \arrow[dr, phantom, "\usebox\pullback", very near start] & Y_m \arrow{d}[right]{F_m}\\
X_{n} \arrow{r}[below]{f} & X_{m}
\end{tikzcd}
};
\end{tikzpicture}
\]

The Cartesian condition above is used below to require that a map $F : Y_\bullet\to X_\bullet$ between simplicial objects conserve the structures determined by the active or inert simplicial maps.

\begin{definition}
A map $F : Y_\bullet \to X_\bullet$ between simplicial spaces is {\em CULF} if $F$ is Cartesian with respect to the active maps in the category $\Delta$. Similarly, a map $F : Y_\bullet \to X_\bullet$ between simplicial spaces is {\em IKEO} if $F$ is Cartesian with respect to the inert maps in the category $\Delta$.
  \end{definition}

The word CULF is a composite acronym combining C and ULF, the two of which
stand for Conservative and Unique Lifting Factorizations, respectively. In some math papers ULF is a common acronym, see \cite[\S 5]{leinster2012}. On the other hand, the acronym IKEO combines IK and EO which are Inner Kan and Equivalence on Objects in long form. 
For an extended discussion of the CULF materials in the proposition below, see \cite[\S 4]{galvez2018}. Some discussion of IKEO recently appeared in \cite[\S 3.1]{GCKJMikeo}.


  

\begin{proposition}\label{culfikeoprop}
Let $F: Y_\bullet \rightarrow X_\bullet$ be a simplicial map. For each $n \geq 1$, let $f_n : [1]\to [n]$ be the active map determined by $f_n(0) := 0$ and $f_n(1) := n$ and set $(Xg)_n := (X(\pi_1), \ldots,X(\pi_n))$ and $(Yg)_n := (Y(\pi_1), \ldots,Y(\pi_n))$, where $\pi_i : [1] \to [n]$ is the inert map determined by $\pi_i(0) := i-1$ and $\pi_i(1) := i$. 
\begin{itemize}
\item The morphism $F$ is CULF if for each $n \in \Z_{\geq 1}$ the following diagram is Cartesian:
\[
\begin{tikzpicture}[baseline= (a).base]
\node[scale=1.0] (a) at (0,0){
\begin{tikzcd}
Y_{1} \arrow{d}[left]{F_1} \arrow[leftarrow]{r}[above]{Y(f_n)} & Y_n \arrow{d}[right]{F_n}\\
X_{1} \arrow[leftarrow]{r}[above]{X(f_n)} & X_n.
\end{tikzcd}
};
\end{tikzpicture}
\]

\item The morphisms $F$ is IKEO if $F_0:Y_0 \rightarrow X_0$ is a (weak) equivalence and for each $n\in\mathbb{Z}_{\geq 1}$ the following diagram is Cartesian:
\[
\begin{tikzpicture}[baseline= (a).base]
\node[scale=1.0] (a) at (0,0){
\begin{tikzcd}
Y_{n} \arrow{d}[left]{F_n} \arrow{r}[above]{(Yg)_n} & Y_1\times_{Y_0} Y_1\times_{Y_0} \cdots \times_{Y_0} Y_1 \arrow{d}[right]{F_1^{\times n}}\\
X_{n} \arrow{r}[above]{(Xg)_n} & X_1\times_{X_0} X_1\times_{X_0} \cdots \times_{X_0} X_1.
\end{tikzcd}
};
\end{tikzpicture}
\]
\end{itemize}
\end{proposition}


\subsection{Groupoids and incidence algebras}\label{incidencesec}
If $I$ is an $\infty$-groupoid, then there is a slice category, or overcategory, $\Sl{I}$ with objects given by maps of $\infty$-groupoids $X\to I$. 
A map $g : I\to J$ induces two maps between slice categories: a pushforward $g_! : \Sl{I}\to\Sl{J}$ and a pullback $g^* : \Sl{J} \to \Sl{I}$. The pushforward $g_{!} : \Sl{I}\to \Sl{J}$ is defined by $g_{!}(a) := a\circ g$ for $a : X\to I$ a groupoid map. The pullback is defined to send $a : X\to J$ to its pullback $g^*(a) : g^*(X)\to I$ over $I$: 
\begin{equation}\label{eq:pullback}
\begin{tikzpicture}[baseline= (a).base]
\node[scale=1.0] (a) at (0,0){
\begin{tikzcd}
g^*(X) \arrow{d}[left]{g^*(a)} \arrow{r}[above]{i} \arrow[dr, phantom, "\usebox\pullback", very near start] & X \arrow{d}[right]{a}\\
I \arrow{r}[above]{g} & J.
\end{tikzcd}
};
\end{tikzpicture}
\end{equation}
In this way, any diagram of $\infty$-categories of the form
\[
\begin{tikzcd}
I \arrow[r, leftarrow, "f"] & M & \arrow[l, leftarrow, swap, "g"] J
\end{tikzcd}
\]
induces maps between associated slice categories,
$$g_{!}\circ f^* : \Sl{I} \to \Sl{J}\quad\quad\textnormal{ and }\quad\quad f_!\circ g^* : \Sl{J} \to \Sl{I}.$$
These maps are used in the definition below to introduce the structural maps of incidence (co)algebras. In addition, we are implicitly using the natural isomorphisms $\Sl{I\times J}\cong \Sl{I}\times \Sl{J}$. 


\begin{definition}\label{incidencealgdef}
\begin{enumerate}
\item If $X_\bullet : \Delta^{\op} \rightarrow \spaces$ is a $2$-Segal space, then the {\em incidence coalgebra} $(\Sl{X_1},\Delta, \epsilon)$ is determined by the spans
\[
\begin{tikzcd}
X_1 & X_2 \arrow{l}[above]{\partial_1} \arrow{r}[above]{(\partial_2,\partial_0)} & X_1\times_{X_0} X_1
\end{tikzcd}
\quad\textnormal{ and }\quad
\begin{tikzcd}
X_1 & X_0 \arrow{l}[above]{s_0} \arrow{r}[above]{\can} & \pt,
\end{tikzcd}
\]
where $\can : X_0\to \pt$ is the canonical map to the one point space $\pt$. In more detail, $\Delta : \Sl{X_1}\to \Sl{X_1\times X_1}$ is given by $\Delta(a) := (\partial_2, \partial_0)_{!}\circ \partial_1^*(a)$ and $\epsilon : \Sl{X_1}\to \cS$ is defined to be $\epsilon(a) := \can_! \circ s_0^*(a)$.

\item If $X_\bullet : \Delta^{\op} \rightarrow \spaces$ is a $2$-Segal space, then the {\em incidence algebra} $(\Sl{X_1}, m, \iota)$ is determined by the adjoint spans: $m(a) := (\partial_1)_!\circ (\partial_2,\partial_0)^*(a)$ and $\iota(a) := (s_0)_! \circ \can^*$. 
\end{enumerate}
  \end{definition}

These definitions can be compared to those found in Section \ref{sec:algSpan}. In more detail, the span which determines the coproduct $\Delta$ and product $m$ above is identical to equation \eqref{eq:multSpan} and the span which determines the counit $\epsilon$ and unit $\iota$ is identical to equation \eqref{eq:unitSpan}. 
The axioms of $2$-Segal spaces allow one to prove that the incidence coproduct $\Delta$ and incidence product $m$ are coassociative and associative, respectively; the arguments involved are the same as arguments found in Section \ref{sec:algSpan}.

Recall that if $F : Y_\bullet \to X_\bullet$ is a map between simplicial objects in $\infty$-groupoids, then there is a map $F([1]) : Y_1 \to X_1$. The proposition below shows that the CULF and IKEO conditions on $F$ translate to homomorphism conditions on the pushforward map $F_! := F([1])_! : \Sl{Y_1} \to \Sl{X_1}$ between slice categories.

\begin{proposition}\label{funprop}
\begin{enumerate}
\item If $F : Y_\bullet \to X_\bullet$ is a CULF map between $2$-Segal spaces, then $F_! : \Sl{Y_1}\to \Sl{X_1}$ is a coalgebra homomorphism.
\item If $F : Y_\bullet \to X_\bullet$ is an IKEO map between $2$-Segal spaces, then $F_! : \Sl{Y_1}\to \Sl{X_1}$ is an algebra homomorphism.
\end{enumerate}
  \end{proposition}
\begin{proof}
For part (1), we follow \cite[Lemma 8.2]{galvez2018}.  First, suppose that 
$F : Y_\bullet \to X_\bullet$ is a CULF map. By virtue of the definitions above there is a commutative diagram
\begin{equation}
\label{eq:funcDiag}
\begin{tikzcd}
Y_{1} \arrow{d}[left]{F_1} \arrow[leftarrow]{r}[above]{Y(f_n)} & Y_n \arrow{d}[right]{F_n} \arrow[rightarrow]{r}{f} & Y_1^{\times n} \arrow{d}[right]{F_1^{\times n}}\\
X_{1} \arrow[leftarrow]{r}[above]{X(f_n)} & X_n \arrow[rightarrow]{r}{f'} & X_1^{\times n}.
\end{tikzcd}
\end{equation}
When $n=0$ the left-hand square is Cartesian because $F$ is assumed to be conservative and when $n \geq 1$ the left-hand square is Cartesian because $F$ is ULF; see Proposition \ref{culfikeoprop} above.

Applying Beck--Chevalley to the diagram \eqref{eq:funcDiag} gives the relation
\begin{equation}\label{lhbeckeq1}
  (F_n)_! \circ Y(f_n)^* \cong X(f_n)^* \circ (F_1)_!.
\end{equation}
Since the $n$-fold coproduct $\Delta_X^{(n)}$ is given by $f'_! \circ X(f_n)$, composing the right-hand side with $f'_!$  gives
\begin{equation}\label{lhbeckeq2}
f'_! \circ X(f_n)^*\circ (F_1)_! \cong \Delta_X^{(n)} \circ (F_1)_!.
\end{equation}
On the other hand, composing the left-hand side with $f'_!$ 
and using $f'_!\circ (F_n)_! \cong (f'\circ F_n)_!$ gives 
\begin{equation}\label{lhbeckeq3}
  (F_1^n\circ f)_!\circ Y(f_n)^* \cong (F^n_1)_!\circ(f_! \circ Y(f_n)^*) \cong (F^n_1)_! \circ \Delta_Y^{(n)},
\end{equation}
where $\Delta_Y^{(n)}$ is the $n$-fold composition of the coproduct from Definition \ref{incidencealgdef}.
Combining equation \eqref{lhbeckeq2} with equations \eqref{lhbeckeq1} and \eqref{lhbeckeq3} implies
$$\Delta_X^{(n)} \circ (F_1)_! \cong (F^n_1)_! \circ \Delta_Y^{(n)}.$$
Hence, $(F_1)_!$ is a homomorphism of incidence coalgebras.

For part (2), there is an analogous argument. From Proposition \ref{culfikeoprop}, the IKEO condition on $F$ implies that the right-hand square of the diagram \eqref{eq:funcDiag} is Cartesian. Applying Beck--Chevalley to this square gives the relation
\begin{equation}\label{rhbeckeq1}
  (F_n)_! \circ f^* \cong f'^* \circ (F_1^{\times n})_!.
\end{equation}
Composing the right-hand side with $X(f_n)_!$ gives 
$$X(f_n)_! \circ \left( f'^* \circ (F_1^{\times n})_!\right) \cong \left( X(f_n)_! \circ f'^*\right) \circ F_1^{\times n})_! \cong  m_X^{(n)} \circ (F_1^{\times n})_!,$$
 where $m_X^{(n)}$ is the $n$-fold composition of the incidence product $m$ in Definition \ref{incidencealgdef}. On the other hand, the left-hand side is
\begin{align*}
X(f_n)_! \circ \left((F_n)_! \circ f^*\right) &\cong \left( X(f_n)\circ F_n \right)_! \circ  f^*\\
&\cong \left( F_1\circ Y(f_n)\right)_! \circ  f^*\\
&\cong (F_1)_! \circ \left( Y(f_n)_! \circ f^*\right)\\
&\cong (F_1)_! \circ m_Y^{(n)}. \qedhere
\end{align*}
\end{proof}

\subsection{Linearizing groupoids}\label{sec:groupoidtolinearalg}

Materials from Section \ref{sec:finCat} admit an extension in which
groupoids are replaced by $\infty$-groupoids. 
A similar argument to Theorem \ref{thm:linearize} gives the derived Hall algebra or numerical incidence algebra in Proposition \ref{prop:theyagree}, see also \cite[\S 2]{toen2006} or \cite[\S 7.8]{galvez2018b}. So a $2$-Segal space $X$ allows one to construct both an incidence algebra and a numerical incidence algebra, Proposition \ref{prop:transition} shows that when suitable finiteness conditions are satisfied there is a functorial transition from the setting of incidence algebras to numerical incidence algebras.


\newcommand{\pro}[1]{\underleftarrow{#1}}
\newcommand{\ind}[1]{\underrightarrow{#1}}
\providecommand{\kat}[1]{\text{{\textsl{#1}}}}
\newcommand{\linn}{\kat{lin}}
\newcommand{\inlin}{\ind\linn}

Call an $\infty$-groupoid $B$ \emph{locally finite} when the groups $\pi_i(B,b)$ are finite for all $b\in B$ and $i\in \mathbb{Z}_{\geq 1}$ and there is an $n\in \mathbb{Z}_{\geq 1}$ such that $\pi_i(B,b)=0$ for all $i>n$.
The \emph{homotopy cardinality} $\|B\|\in \mathbb{Z}$ of a locally finite space $B$ is given by,
$$\|B\| := \sum_{b\in\pi_0(B)} \prod_{i=1}^\infty \vert \pi_i(B,b)\vert^{(-1)^i}.$$
This will allow us to extend the definition of pushforward $f_*$ from Section \ref{sec:finCat}.

To an $\infty$-groupoid $B$ we associate the vector space $\lin(B):= \Fun_c(\pi_0(B), \C)$ of finitely supported functions.
If $f : X\to Y$ is a map between locally finite groupoids and $f$ has finite fibers, $\vert f^{-1}(y)\vert<\infty$ for all $y\in Y$, then there is a pullback map $f^* :  \lin(Y) \to \lin(X)$ given by $f^*(\varphi)(x) := \varphi(f(x))$. 
If $f : X\to Y$ then there is a pushforward map $f_* : \lin(X)\to \lin(Y)$ given by 
$$f_*(\varphi)(y) := \sum_{z\in\pi_0(F_y)} \varphi(i(z))\prod_{i>0} \vert \pi_i(F_y,z)\vert^{(-1)^i}$$
where the map $i : \pi_0(F_y) \to \pi_0(X)$ comes from the inclusion of the homotopy fiber $F_y := \{y\} \times^h_Y X$. Note that any map between locally finite spaces must have locally finite fibers $F_y$. This assignment recovers homotopy cardinality when $\pi : B\to \pt$, by $\pi_*(1_B)=\|B\|$ where $1_B := \sum_{b\in\pi_0(B)} \delta_b$. The pullback $f^*$ and pushforward $f_*$ assignments determine contravariant and covariant functors respectively which satisfy the Beck--Chevalley property.

Denote by $\spaces^{lf}\subset \spaces$ the full subcategory of locally finite spaces and $\Span(\spaces^{lf})^f \subset \Span(\spaces^{lf})$ the subcategory of spans,
$$S \stackrel p\leftarrow M \stackrel q\to T,$$
for which the map $p$ has finite fibers. The assignments in the previous paragraph determine the functor $\lin$ introduced by the proposition below. The proof is by analogy with materials in Section \ref{sec:finCat}, see also \cite[\S 2]{toen2006}.

\begin{proposition}\label{prop:theyagree}
There is a monoidal functor $\lin : \Span(\spaces^{lf})^f \to \Vect_{\C}$ from the category of finite spans of locally finite $\infty$-groupoids to the category of complex vector spaces which is determined by the assignments
$B \mapsto \lin(B)$ and 
$$(S \stackrel p\leftarrow M \stackrel q\to T) \mapsto (q_*p^* : \lin(S) \to \lin(T)).$$
\end{proposition}

By virtue of this proposition, a span $S \stackrel p\leftarrow M \stackrel q\to T$, which determines the map $p^*q_! : \FSl{S} \to \FSl{T}$ from Section \ref{incidencealgdef}, we can assign the map
\[
\lin(S \stackrel p\leftarrow M \stackrel q\to T) := p^*q_* : \lin(S)\to \lin(T).
\]
The next proposition constructs a map $\plin$ from the category of finite slice categories $\FSl{S}$ and spans to $\Vect_{\C}$. If $\FSl{I}\subset \Sl{I}$ is a finite slice category and $a : X\to I$, then set $\plin(a) := a_*(1_X)\in \lin(I)$ where $1_X = \sum_{i\in\pi_0(I)} \delta_i$. Recall from Section \ref{incidencesec} that if $g : I \to J$ is a map of $\infty$-groupoids, then there are maps $g^* : \FSl{J}\to\FSl{I}$ and $g_! : \FSl{I}\to \FSl{J}$. We set $\plin(g^*) := g^* : \lin(J)\to \lin(I)$ and $\plin(g_!) := g_* : \lin(I) \to \lin(J)$.

\begin{proposition}\label{prop:transition}
The assignments above respect maps in the sense that the following properties hold:
\begin{enumerate}
\item If $g : I \to J$ is an $\infty$-groupoid map and $a : X\to I$, then $$\plin(g_!)(a_*(1_X)) = g_*(\plin(a)).$$
\item If $g : I \to J$ is an $\infty$-groupoid map and $a : X\to J$, then $$\plin(g^*(a)) = g^*(\plin(a)).$$
  \end{enumerate}
\end{proposition}
\begin{proof}
For (1): Recall that $g_!(a) = g\circ a : X\to I\to J$. We have made the assignment $\plin(g_! : \FSl{I} \to \FSl{J}) = g_* : \lin(I) \to \lin(J)$, so the map $\plin(g_!) : \lin(I) \to \lin(J)$ takes the vector $a_*(1_X)\in \lin(I)$ to $\plin(g_!)(a_*(1_X)) \in \lin(J)$. Now observe that $\plin(g_!)(a_*(1_X)) = g_*(a_*(1_X)) = (ga)_*(1_X) = \plin(g_*(a)) \in \lin(J)$.

For (2): Recall that $g^*(a) : g^*(X) \to I$ is the pullback in the Cartesian diagram \eqref{eq:pullback}. We have $\plin(g^* : \FSl{J} \to \FSl{I}) = g^* : \lin(J)\to \lin(I)$, so $\plin(g^*)$ takes the vector $a_*(1_X) \in \lin(J)$ to $\plin(g^*)(a_*(1_X)) = g^*(a_*(1_X)) = (g^*a_*)(1_X) \stackrel{(*)}{=} (g^*(a)_* i^*)(1_X) = g^*(a)_*(1_{g^*(X)})$, where the map $i : g^*(X) \to J$ is the inclusion $i : g^*(X) \to X$ and the Beck--Chevalley property is used to justify $(*)$.
  \end{proof}

For other details see \cite{galvez2018b, GCKJMhtylinearalg}. 

In this way the construction of the numerical Hall algebra factors through incidence algebras when the span defining the product satisfies sufficient finiteness conditions.




\subsection{From incidence algebras to Hall algebras}
\renewcommand{\AA}{\mathcal{A}}
\renewcommand{\BB}{\mathcal{B}}





Since the Waldhausen construction $\Wald_{\bullet}(\CC)$ is $2$-Segal (Theorem \ref{thm:Wald2Seg}), there is an incidence coalgebra and an incidence algebra,
$$(\Sl{\Wald_1(\CC)}, \Delta, \epsilon)\quad\quad\textnormal{ and }\quad\quad (\Sl{\Wald_1(\CC)}, m, \iota).$$
The Hall (co)algebra introduced in Section \ref{sec:finCat} can be obtained from the incidence (co)algebra by the groupoid linearization procedure of Section \ref{sec:groupoidtolinearalg} above. This allows us to translate Proposition \ref{funprop} into a statement about functoriality of the Hall (co)algebra. An exact functor $F : \AA\to \BB$ induces a map $F : S_\bullet(\AA)\to S_\bullet(\BB)$ between Waldhausen constructions. We analyze what the CULF and IKEO conditions on $F$ entail. This is the subject of the theorem below.

\begin{theorem}\label{FunSummary}
  Suppose that $F : \AA\to \BB$ is an exact finitary functor between finitary abelian categories. The corresponding map $F_\bullet : \Wald_\bullet(\AA)\to \Wald_\bullet(\BB)$ between $2$-Segal spaces induces homomorphisms between corresponding Hall coalgebras or Hall algebras when $F$ is CULF or IKEO, respectively. These conditions correspond to
  \begin{itemize}
  \item CULF: If $F(Y)\in Ob(\BB)$ is in the image of $F$, then so must be all of the subobjects and quotients of $F(Y)$. That is, for all $X' \subset F(Y)$, there exists $X$ in $Ob(\AA)$ such that $F(X) = X'$ and there exists $Y/X$ such that $F(Y/X) = F(Y)/X'$.
    \item IKEO: If $F(Z)$ and $F(X)$  are in the image of $F$ and $Y'$ is an extension of $F(Z)$ by $F(X)$ in $\BB$, then there is a unique extension $Y$ of $Z$ by $X$ in $\AA$ such that $F(Y) = Y'$.
    \end{itemize}
  \end{theorem}

\begin{proof}
We examine the special $n=2$ cases of the CULF and IKEO conditions. This suffices to determine the conditions in our premises, since the data of the Hall (co)algebra involves only the $2$-truncation of $\Wald_{\bullet}$. Conversely, it is easy to check that these conditions suffice to determine maps between Hall (co)algebras after appropriate finiteness conditions are imposed.

When $n=2$, the CULF condition as spelled out in Proposition \ref{culfikeoprop} requires that the diagram
\[
\begin{tikzpicture}[baseline= (a).base]
\node[scale=1.0] (a) at (0,0){
\begin{tikzcd}
\Wald_{1}(\AA) \arrow{d}[left]{F_1} \arrow[leftarrow]{r}[above]{\partial_1} & S_2(\AA) \arrow{d}[right]{F_2}\\
\Wald_{1}(\BB) \arrow[leftarrow]{r}[above]{\partial_1} & \Wald_2(\BB)
\end{tikzcd}
};
\end{tikzpicture}
\]
is Cartesian. So there is an isomorphism
\[
\begin{tikzcd}
\gamma : \Wald_2(\AA) \arrow{r}[above]{\sim} & \Wald_1(\AA)\times_{\Wald_1(\BB)} \Wald_2(\BB).
\end{tikzcd}
\]
This amounts to the requirement that the association
\begin{equation}
\label{eq:FonS2}
\begin{array}{r@{}r@{}l}
    F_2 : \Wald_2(\AA) &{} \rightarrow &{} \Wald_2(\BB)\\
    0 \to X\xrightarrow{f} Y\to Z\to 0 &{} \mapsto &{} 0 \to F(X)\to F(Y)\to F(Z)\to 0,
\end{array}
\end{equation}
where $Z := \coker(f)$, induced by the exact functor $F$ must satisfy the following additional property: If
$$0\to \gamma(X) \to F(Y)\to \gamma(Z)\to 0$$
is a short exact sequence in $\BB$, then there exist $X$ and $Z$ in $S_1(\AA)$ such that $F_1(X) = \gamma(X)$ and $F_1(Z)=\gamma(Z)$. In words, the image of $F$ is closed under subobjects and quotients.

On the other hand, the IKEO condition from Proposition \ref{culfikeoprop} translates to the requirement that the diagram
\[
\begin{tikzpicture}[baseline= (a).base]
\node[scale=1.0] (a) at (0,0){
\begin{tikzcd}
\Wald_{2}(\AA) \arrow{d}[left]{F_2} \arrow{r}[above]{(\partial_2,\partial_0)} & \Wald_1(\AA)\times \Wald_1(\AA) \arrow{d}[right]{F_1\times F_1}\\
\Wald_{2}(\BB) \arrow{r}[above]{(\partial_2,\partial_0)} & \Wald_1(\BB)\times \Wald_1(\BB)
\end{tikzcd}
};
\end{tikzpicture}
\]
is Cartesian. (Note that $\Wald_0(\AA)$ and $\Wald_0(\BB)$ are both the trivial groupoid.) So there is an isomorphism
\[
\begin{tikzcd}
\gamma : \Wald_2(\AA) \arrow{r}[above]{\sim} & \Wald_2(\BB)\times_{\Wald_1(\BB)\times \Wald_1(\BB)} \left( \Wald_1(\AA)\times \Wald_1(\AA)\right)
\end{tikzcd}
\]
which translates to the following additional requirement on the morphism from equation \eqref{eq:FonS2}:
If
$$0\to X\to Y\to Z\to 0$$ 
is a short exact sequence in $\AA$, then there is a unique extension of $F(Z)$ by $F(X)$ in $\BB$ which is $F(Y)$.
\end{proof}

A slightly more involved approach would study the CULF and IKEO conditions 
only after applying homotopy cardinality to all of the diagrams involved. This might lead to slightly different statements; compare to \cite[Proposition 5.3]{cooperSamuelson2022}.

\section{Representations of Hall algebras via relative $2$-Segal spaces}
\label{sec:rel2Seg}

The focus of this section is the theory of relative $2$-Segal spaces, as formulated in \cite{walde2016,mbyoung2018b}. From the perspective of this paper, the main feature of relative $2$-Segal spaces is that they lead naturally to modules over Hall algebras.

\subsection{Relative $2$-Segal spaces}

\begin{definition}
Let $X_{\bullet}$ be a $2$-Segal space. A morphism $F_{\bullet}: Y_{\bullet} \rightarrow X_{\bullet}$ of simplicial spaces is \emph{relative $2$-Segal} if $Y_{\bullet}$ is $1$-Segal and, for every $n \geq 2$ and $0 \leq i < j \leq n$, the outside square of the diagram
\begin{equation}
\label{eq:rel2SegalDiagram}
\begin{tikzpicture}[baseline= (a).base]
\node[scale=1] (a) at (0,0){
\begin{tikzcd}
Y_n \arrow{r} \arrow{d} & Y_{\{0, \dots, i,j, \dots,n\}} \arrow{d}  \\
Y_{\{i, \dots, j\}} \arrow{r} \arrow{d}[left]{F_{\{i, \dots, j\}}} & Y_{\{i,j\}} \arrow{d}[right]{F_{\{i,j\}}} \\
X_{\{i, \dots, j\}} \arrow{r} & X_{\{i,j\}}
\end{tikzcd}
};
\end{tikzpicture}
\end{equation}
is Cartesian.
\end{definition}

As for $2$-Segal spaces, we have formulated the relative $2$-Segal condition in a manner which makes their relevance to modules over Hall algebras most apparent. More intrinsic formulations can be found in \cite{walde2016,mbyoung2018b}.

We give two classes of examples of relative $2$-Segal spaces, both of which are relative variants of the Waldhausen $\Wald_{\bullet}$-construction of Section \ref{sec:waldhausen}.

\begin{example}[{\cite{walde2016}, \cite{mbyoung2018b}}]
\label{ex:framedConstr}
Let $\CC$ be a $k$-linear abelian category with Grothendieck group $K_0(\CC)$. Following Bridgeland, a \emph{stability function} on $\CC$ is a group homomorphism $Z : K_0(\CC) \rightarrow \mathbb{C}$ such that
\[
Z(A) \in \mathbb{H}_+ := \left\{ m e^{\sqrt{-1}\pi  \phi} \mid m \in \R_{>0}, \; \phi \in (0,1] \right\} \subset \mathbb{C}
\]
for all non-zero objects $A \in \CC$ \cite[\S 2]{bridgeland2007}. Write $\phi(A) \in (0,1]$ for the phase of $Z(A)$. A non-zero object $A \in \CC$ is called \emph{$Z$-semistable} if $\phi(A^{\prime}) \leq \phi(A)$ for all non-zero subobjects $A^{\prime} \subset A$. The full subcategory $\CC^{\sst}_{\phi} \subset \CC$ of $Z$-semistable objects of phase $\phi \in (0,1]$ (together with the zero object) is abelian whence defines a $2$-Segal space $\Wald_{\bullet}(\CC^{\sst}_{\phi})$ by Theorem \ref{thm:Wald2Seg}.

Let $\Phi: \CC \rightarrow \vect_k$ be a left exact functor. A \emph{framed object} is then a pair $(M,s)$ consisting of an object $M \in \CC$ and a section $s \in \Phi(M)$ \cite[\S 4]{soibelman2016}. A morphism of framed objects $(M,s) \rightarrow (M^{\prime}, s^{\prime})$ is a pair $(\pi, \lambda) \in \Hom_{\CC}(M, M^{\prime}) \times k$ which satisfies $\Phi(\pi)(s) = \lambda s^{\prime}$. A framed object $(M,s)$ is called \emph{stable framed} if $M$ is $Z$-semistable and $\phi(A) < \phi(M)$ for all non-zero proper subobjects $A \subset M$ for which $s \in \Phi(A) \subset \Phi(M)$.

Let $\Wald^{\st \mhyphen \fr}_{\bullet}(\CC^{\sst}_{\phi})$ be the simplicial groupoid whose $n$-simplices consist of diagrams of the form
\[
\begin{tikzpicture}[baseline= (a).base]
\node[scale=1] (a) at (0,0){
\begin{tikzcd}
0 \arrow[tail]{r} & A_{\{0,1\}} \arrow[two heads]{d} \arrow[tail]{r} & \cdots   \arrow[tail]{r} & A_{\{0,n\}} \arrow[two heads]{d} \arrow[tail]{r} & (M_0, s_0) \arrow[two heads]{d} \\
 & 0 \arrow[tail]{r} & \cdots   \arrow[tail]{r} & A_{\{1,n\}} \arrow[two heads]{d}  \arrow[tail]{r} & (M_1, s_1) \arrow[two heads]{d} \\
 &  & \ddots & \vdots \arrow[two heads]{d} & \vdots \arrow[two heads]{d} \\
  &  & & 0  \arrow[tail]{r} & (M_n, s_n) \arrow[two heads]{d} \\
  &  &  &  & 0
\end{tikzcd}
};
\end{tikzpicture}
\]
where each $(M_i,s_i)$ is a stable framed object and the underlying diagram, obtained by forgetting the framing data $s_0, \dots, s_n$, is an object of $\Wald_{n+1}(\CC^{\sst}_{\phi})$. In the above diagram a morphism $A_{\{i,n\}} \rightarrow (M_i,s_i)$ is simply a morphism $A_{\{i,n\}} \rightarrow M_i$.
The simplicial morphism $F_{\bullet} :\Wald^{\st \mhyphen \fr}_{\bullet}(\CC^{\sst}_{\phi}) \rightarrow \Wald_{\bullet}(\CC^{\sst}_{\phi})$ which forgets the rightmost column is relative $2$-Segal.
\end{example}

\begin{example}[{\cite{mbyoung2018b}}]
\label{ex:RConstr}
Let $\CC$ be a proto-exact category with duality $(P,\Theta)$, that is, a proto-exact duality functor $P : \CC^{\op} \rightarrow \CC$ and a natural double dual isomorphism $\Theta: \id_{\CC} \Rightarrow P \circ P^{\op}$ whose components satisfy $P(\Theta_U) \circ \Theta_{P(U)} = \id_{P(U)}$, $U \in \CC$. A \emph{symmetric form} in $\CC$ is then a pair ($M,\psi_M)$ consisting of an object $M \in \CC$ and an isomorphism $\psi_M: M \xrightarrow[]{\sim} P(M)$ which satisfies $P(\psi_M) \circ \Theta_M = \psi_M$. The \emph{hermitian groupoid} $\CC^h$ is the groupoid of symmetric forms and their isometries.

Let $(N,\psi_N)$ be a symmetric form in $\CC$. The \emph{orthogonal} $U^{\perp}$ of an inflation $i : U \rightarrowtail N$ is defined to be the pullback
\[
\begin{tikzpicture}[baseline= (a).base]
\node[scale=1] (a) at (0,0){
\begin{tikzcd}
U^{\perp} \arrow[dashed,two heads]{d} \arrow[dashed,tail]{r}{k} & N \arrow[two heads]{d}{P(i) \circ \psi_N} \\
0 \arrow[tail]{r} & P(U).
\end{tikzcd}
};
\end{tikzpicture}
\]
The inflation $i$ is called \emph{isotropic} if $P(i) \circ \psi_N \circ i$ is zero and the induced morphism $U \rightarrow U^{\perp}$ is an inflation. In this case, if $\CC$ is exact, then on the pushout
\[
\begin{tikzpicture}[baseline= (a).base]
\node[scale=1] (a) at (0,0){
\begin{tikzcd}
U \arrow[two heads]{d} \arrow[tail]{r} & U^{\perp} \arrow[dashed,two heads]{d}{\pi} \\
0 \arrow[dashed,tail]{r} & M
\end{tikzcd}
};
\end{tikzpicture}
\]
there exists a unique symmetric form $\psi_M$ which satisfies $P(k) \psi_N k = P(\pi) \psi_M \pi$. All known proto-exact categories are known to satisfy this property, which we henceforth assume. We call $(M, \psi_M)$ is the \emph{isotropic reduction of $N$ by $U$} and denote it by $N \git U$. 


%
%

Following Shaprio and Yao \cite{shapiro1996}, let $\RWald_{\bullet}(\CC)$ be the simplicial groupoid whose $n$-simplicies $\RWald_n(\CC)$ consist of diagams
\begin{equation}
\label{eq:RWaldDiag}
\begin{tikzpicture}[baseline= (a).base]
\node[scale=0.75] (a) at (0,0){
\begin{tikzcd}
0 \arrow[tail]{r} & A_{\{0,1\}} \arrow[two heads]{d} \arrow[tail]{r} & \cdots \arrow[tail]{r}& A_{\{0,n\}} \arrow[tail]{r} \arrow[two heads]{d} & A^{\perp}_{\{0,n\}} \arrow[tail]{r} \arrow[two heads]{d} & \cdots \arrow[tail]{r} & A^{\perp}_{\{0,1\}} \arrow[two heads]{d} \arrow[tail]{r} & (M_0, \psi_0) \arrow[two heads]{d} \\
 & 0 \arrow[tail]{r} & \cdots \arrow[tail]{r} & A_{\{1,n\}} \arrow[tail]{r} & A^{\perp}_{\{1,n\}} \arrow[tail]{r} & \cdots \arrow[tail]{r} & (M_1, \psi_1) \arrow[two heads]{d}  \arrow[tail]{r} & P(A^{\perp}_{\{0,1\}}) \arrow[two heads]{d} \\
 &  & & & & \ddots & \vdots \arrow[two heads]{d} & \vdots \arrow[two heads]{d} \\
 &  & & & & & P(A_{\{1,n\}}^{\perp}) \arrow[two heads]{d} \arrow[tail]{r} & P(A_{\{0,n\}}^{\perp}) \arrow[two heads]{d} \\
 &  & & & & & P(A_{\{1,n\}}) \arrow[two heads]{d} \arrow[tail]{r} & P(A_{\{0,n\}}) \arrow[two heads]{d} \\
 &  & & & & & \vdots \arrow[two heads]{d} & \vdots \arrow[two heads]{d} \\
 &  & & & & & 0 \arrow[tail]{r} & P(A_{\{0,1\}}) \arrow[two heads]{d} \\
 &  &  & & & & & 0
\end{tikzcd}
};
\end{tikzpicture}
\end{equation}
such that 
\begin{enumerate}[label=(\roman*)]
\item each object $(M_i, \psi_i)$ is a symmetric form,

\item upon forgetting $\psi_i$, $0 \leq i \leq n$, the diagram is an object of $\Wald_{2n+1}(\CC)$,

\item each inflation $A_{\{ i,j \}} \rightarrowtail (M_i,\psi_i)$ is isotropic with orthogonal $A^{\perp}_{\{ i,j \}}\rightarrowtail (M_i,\psi_i)$,

\item for every $0 \leq i \leq j \leq n$, the symmetric form $(M_j,\psi_j)$ is isometric to the reduction $(M_i, \psi_i) \git A_{\{i,j\}}$,

\item the diagram below the diagonal $x=y$ is the $P$-image of the diagram above the diagonal.
\end{enumerate}
For example, $\RWald_0(\CC) \simeq \CC^h$ while $\RWald_1(\CC)$ consists of diagrams
\begin{equation}
\label{eq:R1Diag}
\begin{tikzpicture}[baseline= (a).base]
\node[scale=1] (a) at (0,0){
\begin{tikzcd}
0 \arrow[tail]{r} & A_{\{0,1\}} \arrow[two heads]{d} \arrow[tail]{r} & A_{\{0,1^{\prime}\}} \arrow[tail]{r} \arrow[two heads]{d} & (M_0,\psi_0) \arrow[two heads]{d} \\
 & 0 \arrow[tail]{r} & (M_1,\psi_1)  \arrow[two heads]{d} \arrow[tail]{r} & A_{\{1,0^{\prime}\}} \arrow[two heads]{d} \\
 & & 0 \arrow[tail]{r} & A_{\{1^{\prime},0^{\prime}\}} \arrow[two heads]{d} \\
 & & & 0
\end{tikzcd}
};
\end{tikzpicture}
\end{equation}
which present $(M_1,\psi_1)$ as the isotropic reduction of $(M_0,\psi_0)$ by $A_{\{0,1\}}$.
%
%
In general, objects of $\RWald_n(\CC)$ are isotropic $n$-flags together with presentations of all subquotients and subreductions.

The forgetful map $F_{\bullet}: \RWald_{\bullet}(\CC) \rightarrow \Wald_{\bullet}(\CC)$ which sends a diagram \eqref{eq:RWaldDiag} to the subdiagram on the objects $\{ A_{\{p,q\}}\}_{0 \leq p \leq q \leq n}$ is relative $2$-Segal. The proof boils down to a symmetric form of the Third Isomorphism Theorem, asserting a canonical isometry
\[
(M,\psi_M) \git U \simeq \left( (M,\psi_M) \git V\right)  \git (U \slash V)
\]
for a chain of isotropic subobjects $V \rightarrowtail U \rightarrow (M,\psi_M)$. In the context of $\infty$-categories, the relative $2$-Segal property of $F_{\bullet}$ was proven recently by G\"{o}dicke \cite{godicke2024}.

The $\RWald_{\bullet}$-construction was originally introduced in the context of Grothendieck--Witt theory. Indeed, for an exact category $\CC$, the Grothendieck--Witt group $GW_i(\CC)$ is the $i$\textsuperscript{th} homotopy group of the homotopy fibre over $0 \in \CC$ of the map $\vert F_{\bullet} \vert: \vert \RWald_{\bullet}(\CC) \vert  \rightarrow \vert \Wald_{\bullet}(\CC) \vert$ \cite{schlichting2010b}. For studies of the (Grothendieck--)Witt theory of proto-exact categories, see \cite{eberhardtLorscheidYoung2022b,eberhardtLorscheidYoung2022}.
\end{example}

\subsection{Hall modules via the $\RWald_{\bullet}$-construction}

Let $F_{\bullet}: Y_{\bullet} \rightarrow X_{\bullet}$ be relative $2$-Segal. Arguing as in Section \ref{sec:algSpan}, the span
\[
\begin{tikzpicture}[baseline= (a).base]
\node[scale=1.0] (a) at (0,0){
\begin{tikzcd}
{} & Y_1 \arrow{dl}[above left]{(F_1,\partial_0)} \arrow{dr}[above right]{\partial_1}& {} \\
X_1 \times Y_0 & {} & Y_0
\end{tikzcd}
};
\end{tikzpicture}
\]
gives $Y_0$ the structure of a left $(X_1,m)$-module object in $\Span(\spaces)$. Specifically, the first non-trivial $1$-Segal condition on $Y_{\bullet}$ and the first non-trivial instance of diagram \eqref{eq:rel2SegalDiagram} imply associativity of the $(X_1,m)$-action.

Consider again the setting of Section \ref{sec:finCat}, so that the proto-exact category $\CC$ is finitary and $\lin = \Fun_c(\pi_0(-),\C)$. We consider the relative $2$-Segal space $\RWald_{\bullet}(\CC) \rightarrow \Wald_{\bullet}(\CC)$ of Example \ref{ex:RConstr}. Then $M(\CC) := M(\RWald_{\bullet}(\CC); \lin)$ is the complex vector space with basis $\{1^P_W\}_{W \in \pi_0 \CC^h}$ labelled by isometry classes of symmetric forms in $\CC$. The analogue of Lemma \ref{lem:finImpliesSmPr} holds, so that the morphisms $(F_1,\partial_0)$ and $\partial_1$ which appear in the span
\[
\begin{tikzpicture}[baseline= (a).base]
\node[scale=0.8] (a) at (0,0){
\begin{tikzcd}[,column sep=0.25in]
{} & \RWald_1(\CC) \arrow{dl}[above left]{(F_1,\partial_0)} \arrow{dr}[above right]{\partial_1}& {} \\
\Wald_1(\CC) \times \RWald_0(\CC) & {} & \RWald_0(\CC)
\end{tikzcd}
};
\end{tikzpicture},
\qquad
\begin{tikzpicture}[baseline= (a).base]
\node[scale=0.8] (a) at (0,0){
\begin{tikzcd}[arrows={|->},column sep=0.2in]
{} & \textnormal{Diag. } \eqref{eq:R1Diag} \arrow{dl} \arrow{dr} & {} \\
(U,(M_1,\psi_1)) & {} & (M_0,\psi_0)
\end{tikzcd}
};
\end{tikzpicture}
\]
are smooth and proper with repsect to $\lin$, respectively. Hence, $M(\CC)$ becomes a left $\Hall(\CC)$-module with action determined by
\[
1_U \star 1^P_M = \sum_{N \in \pi_0 \CC^h} d_{U,M}^N 1^P_N.
\]
The structure constant $d_{U,M}^N$ is equal to the cardinality of the set 
\[
\{U^{\prime} \subset N \mid U^{\prime} \simeq U \mbox{ is isotropic in }V, \; N \git U^{\prime} \simeq M \}.
\]

\begin{example}
\label{ex:pointFqMod}
We consider a symplectic counterpart of Example \ref{ex:pointFq}. Give $\vect_{\F_q}$ the standard $\F_q$-linear duality functor $P = \Hom_{\F_q}(-, \F_q)$ and double dual isomorphism $\Theta_U = - \ev_U$ the negative of the canonical evaluation isomorphism. Symmetric forms are then symplectic vector spaces. The resulting Hall module is
\[
\HallM(\vect_{\F_q}) = \bigoplus_{m = 0}^{\infty} \C \cdot 1_{2m}^P,
\]
where $1_{2m}^P$ denotes the characteristic function of the symplectic vector space of dimension $2m$. The module structure is
\[
1_n \star 1_{2m}^P = q^{\frac{n(n+1)}{2}}\left[ \begin{matrix} n+m \\ n \end{matrix} \right]_q 1^P_{n+m},
\]
where now the structure constant counts the number of $\F_q$-rational points of the isotropic Grassmannian $\textnormal{IGr}(n,2(m+n))$.

Taking instead $\Theta_U = \ev_U$ results in symmetric forms which are orthogonal vector spaces. The resulting Hall module $\HallM(\vect_{\F_q})$ is a direct sum of $\Hall(\vect_{\F_q})$-modules labelled by the Witt group, that is, the parity of the dimension of orthogonal vector spaces, and has structure constants which count $\F_q$-rational points of orthogonal Grassmannians.
\end{example}

\begin{example}[\cite{mbyoung2016}]
\label{ex:quiverFqMod}
We consider a relative counterpart of Example \ref{ex:quiverFq}. Let $\sigma$ be an anti-involution of a quiver $Q$, that is, $\sigma$ is self-bijections of the sets of vertices and arrows of $Q$ such that if $\alpha: i \rightarrow j$ is an arrow in $Q$, then so too is $\sigma(\alpha): \sigma(j) \rightarrow \sigma(i)$. Define a duality structure on $\rep_{\F_q}(Q)$ by $P = (-)^{\vee} \circ \sigma^*$, where $(-)^{\vee}$ denotes vertex-wise $\F_q$-linear duality and $\Theta$ is the vertex-wise evaluation isomorphism. Symmetric forms in $\rep_{\F_q}(Q)$ are orthogonal counterparts of quiver representations. See \cite{derksen2002}. More general duality structures exist which incorporate various signs associated to vertices and arrows and give symmetric forms which are mixed orthogonal/symplectic representations of $Q$. By incorporating the coalgebra and comodule structures of $H(Q)$ and $M(Q)$ (the latter has been neglected in this article), one checks that the $H(Q)$-module $M(Q)$ lifts to a module over Kashiwara--Enomoto's \emph{reduced $\sigma$-analogue} of the quantum group $U_{\sqrt{q}}(\mathfrak{g}_Q)$ \cite{enomoto2008}. With this interpretation, $M(Q)$ decomposes into highest weight modules labelled by the cuspidal elements of $M(Q)$.
\end{example}

The $\F_1$-analogue of Example \ref{ex:quiverFqMod} is studied in \cite{mbyoung2021}.

\begin{example}[\cite{mbyoung2020b}]
\label{ex:CoHM}
We consider a relative counterpart of Section \ref{sec:CoHA}. Give $\rep_{\C}(Q)$ a duality structure as in Example \ref{ex:quiverFqMod}. The resulting \emph{cohomological Hall module} $M(Q) = M(\RWald_{\bullet}(\rep_{\C}(Q)); H^{\bullet}(-;\Q))$ plays a central role in the orientifold Donaldson--Thomas theory of $Q$, that is, the mixed orthogonal/symplectic generalization of Donaldson--Thomas theory \cite{mbyoung2015,mbyoung2020b}. Again, we restrict attention to the $m$-loop quiver $Q_m$. Fix the duality $P= \Hom_{\C}(-,\C)$ and double dual isomorphism $\Theta = - \ev$. There is an isomorphism of stacks
\[
\RWald_0(\rep_{\C}(Q_m))
\simeq
\bigsqcup_{e = 0}^{\infty}
[\mathfrak{sp}_{2e}(\C)^{\oplus m} \slash \textnormal{Sp}_{2e}(\C)]
\]
so that the underlying $\Z \times \Z$-graded vector space of $M(Q_m)$ is
\[
M(Q_m)
=
\bigoplus_{e=0}^{\infty} M_{2e}(Q_m),
\]
where
\[
M_{2e}(Q_m) \simeq H^{\bullet}( B\textnormal{Sp}_{2e}(\C))
\simeq 
\Q[z_1^2,\dots,z_e^2]^{\mathfrak{S}_e},
\qquad
\vert z_i \vert =2.
\]
Keeping the notation of Section \ref{sec:CoHA}, the action of $f \in H_{d}(Q_m)$ on $g \in M_{2e}(Q_m)$ is
\begin{align*}
f \star g
=
2^{(m-1)d} \sum_{ \pi \in \mathfrak{sh}_{d,e}^{\sigma}}  \pi  \Big[ & f   (x_1, \dots, x_d) g(z_1, \dots, z_e) \times \\ & \Big(\prod_{i=1}^d (-x_i)   \prod_{1 \leq i < j \leq d} (- x_i - x_j)  \prod_{i=1}^d \prod_{j=1}^{e} (x_i^2 - z_j^2) \Big)^{m-1} \Big],
\end{align*}
where $\mathfrak{sh}^{\sigma}_{d,e} = \mathbb{Z}_2^d \times \mathfrak{sh}_{d, e}$ is the set of signed-shuffles of type $(d,e)$. There is an analogue of Efimov's theorem in this relative setting. Namely, define a $2 \Z_{\geq 0} \times \Z$-graded vector space by
\[
W_{Q_m}^{\prim} = M(Q_M) \slash (H(Q_m)_+ \star M(Q_m)),
\]
where $H(Q_m)_+$ is the augmentation ideal of $H(Q_m)$, and pick a graded splitting so as to view $W^{\prim}_{Q_m}$ as a subspace of $M(Q_m)$. The vector space $W_{Q_m}^{\prim}$ is the \emph{cohomological orientifold Donaldson--Thomas invariant of $Q$}. Then, for each $e \in 2 \Z_{\geq 0}$, there is a $\Z_{\geq 0} \times \mathbb{Z}$-graded subalgebra $H(Q_m)(e) \subset H(Q_m)$ such that the action map
\begin{equation}
\label{eq:oriDTConj}
\star: \bigoplus_{e \in 2 \Z_{\geq 0}} H(Q_m)(e) \otimes_{\Q} W^{\prim}_{Q_m,e} \rightarrow M(Q_m)
\end{equation}
is an isomorphism of graded vector spaces. Moreover, $H(Q_m)(e)$ can be described with sufficient understanding of $V_{Q_m}^{\prim}$. This isomorphism can be used to inductively determine $W^{\prim}_{Q_m,e}$. While the isomorphism  \eqref{eq:oriDTConj} remains conjectural for general symmetric quivers, it is known that $W_{Q}^{\prim}$ is isomorphic to the Chow group of the moduli space of stable symmetric representations of $Q$ \cite{franzenYoung2018}.
\end{example}

There is also a motivic version of Example \ref{ex:CoHM}; see \cite{mbyoung2015,mbyoung2021,bu2023}.

\begin{example}
Working in the context of Example \ref{ex:greenZele}, consider $\vect_{\F_q}$ with the standard orthogonal or symplectic duality structure, as described in Example \ref{ex:pointFqMod}. The resulting module $M(\RWald_{\bullet}(\vect_{\F_q});\lin)$ is isomorphic to
\[
\bigoplus_{m=0}^{\infty} K_0(\rep_{\C}(G_m(\F_q))),
\]
where $G_m = \textnormal{O}_m$ (resp. $G_m = \textnormal{Sp}_{2m}$) if $\Theta = \ev$ (resp. $\Theta = -\ev$) with $H(\Wald_{\bullet}(\vect_{\F_q});\lin)$ acting by parabolic induction. van Leeuwen proved that $M(\RWald_{\bullet}(\vect_{\F_q});\lin)$ is minimally generated over $H(\Wald_{\bullet}(\vect_{\F_q});\lin)$ by the set of cuspidal representations of the tower of groups $G_m(\F_q)$, $m \geq 0$ \cite{leeuwen1991}.
\end{example}

\subsection{Hall modules via framing}

\begin{example}
We work with the stable framed $\Wald_{\bullet}$-construction of Example \ref{ex:framedConstr}. Subject to smoothness and properness conditions, for each phase $\phi \in [0,1)$ there results a module $M^{\st \mhyphen \fr}_{\phi}(\CC) := M(\Wald_{\bullet}^{\st \mhyphen \fr}(\CC^{\sst}_{\phi});\lin)$ over the Hall algebra $H^{\sst}_{\phi}(\CC) := H(\Wald_{\bullet}(\CC^{\sst}_{\phi});\lin)$.

For example, let $\CC=\rep_{\C}(Q)$. Sending a representation to its dimension vector defines a surjective group homomorphism $K_0(\rep_{\C}(Q)) \rightarrow \mathbb{Z}^{Q_0}$. A tuple $\zeta =(\zeta_i)_{i \in Q_0} \in \mathbb{H}_+^{Q_0}$ defines a function $\mathbb{Z}^{Q_0} \rightarrow \mathbb{C}, d \mapsto \sum_{i \in Q_0} \zeta_i d_i$, which induces a stability function $Z$ on $\rep_{\C}(Q)$. For any $f \in \mathbb{Z}_{\geq 0}^{Q_0}$, the functor $\Phi: U \mapsto \bigoplus_{i \in Q_0} \Hom_{\C}(\C^{f_i}, U_i)$ is a framing. Working in the context of Section \ref{sec:CoHA} and Example \ref{ex:CoHM}, the resulting cohomological $H^{\sst}_{\phi}(Q)$-modules $M^{\st \mhyphen \fr}_{\phi}(Q)$ play an important role in Donaldson--Thomas theory \cite{soibelman2016,franzen2016}. With knowledge of the structure of the algebra $H^{\sst}_{\phi}(Q)$, for example as described in Section \ref{sec:CoHA}, the module $M^{\st \mhyphen \fr}_{\phi}(Q)$ can be used to inductively determine topological invariants of the smooth moduli space of stable framed representations of $Q$, also known as non-commutative Hilbert schemes. From a different perspective, $M^{\st \mhyphen \fr}_{\phi}(Q)$ can be seen to approximate $H^{\sst}_{\phi}(Q)$ in the limit $f \rightarrow \infty$. This has the technical benefit that, for sufficiently large $f$, the underlying vector space of $M^{\st \mhyphen \fr}_{\phi}(Q)$ is the cohomology of a smooth variety, as opposed to a stack. This idea is used to great effect to study $H^{\sst}_{\phi}(Q)$ in \cite{davison2020}. Similar ideas can be applied in the context of moduli spaces of sheaves over a smooth projective curve \cite{mozgovoy2015} or surface \cite{diaconescu2022}.
\end{example}


\newcommand{\etalchar}[1]{$^{#1}$}
\providecommand{\bysame}{\leavevmode\hbox to3em{\hrulefill}\thinspace}
\providecommand{\MR}{\relax\ifhmode\unskip\space\fi MR }
\providecommand{\MRhref}[2]{%
  \href{http://www.ams.org/mathscinet-getitem?mr=#1}{#2}
}
\providecommand{\href}[2]{#2}

\end{document}